\documentclass[12pt, A4]{article}
\usepackage{amsmath,amssymb,amsfonts,mathrsfs,hyperref,microtype, amsthm}
\usepackage{bm, enumerate}
\usepackage{geometry}
\usepackage{eufrak}
\usepackage{graphicx,color}
\usepackage[T1]{fontenc}
\usepackage{rotating}
\usepackage{booktabs}
\usepackage{mathtools}
\usepackage{float}
\usepackage{breqn}
\usepackage{graphicx}
\usepackage{caption}

\newtheorem{thm}{Theorem}[section]
\newtheorem{rem}{Remark}[section]
\newtheorem{definition}{Definition}[section]
\newtheorem{lem}{Lemma}[section]
\newtheorem{prop}{Proposition}[section]
\newtheorem{cor}{Corollary}[section]
\newtheorem{assu}{Assumption}
\newtheorem{ex}{Example}[section]

\numberwithin{equation}{section}

\def\HH{ \EuFrak H}
\def\N{{\rm I\kern-0.16em N}}
\def\R{{\rm I\kern-0.16em R}}
\def \E{{\rm I\kern-0.16em E}}
\def\P{{\rm I\kern-0.16em P}}
\def\F{{\rm I\kern-0.16em F}}
\def\B{{\rm I\kern-0.16em B}}
\def\C{{\rm I\kern-0.46em C}}
\def\G{{\rm I\kern-0.50em G}}

\newcommand{\ud}{\mathrm{d}}

\numberwithin{equation}{section}

\font\eka=cmex10
\usepackage{ae}

\def\ind{\mathrel{\hbox{\rlap{%
\hbox to 7.5pt{\hrulefill}}\raise6.6pt\hbox{\eka\char'167}}}}
\begin{document}

\title{\Large{\bf Stein's method on the second Wiener chaos :
    2-Wasserstein distance }}

\date{}
\renewcommand{\thefootnote}{\fnsymbol{footnote}}

\author{Benjamin Arras, Ehsan Azmoodeh, Guillaume Poly and Yvik Swan}

\maketitle

\abstract 

In the first part of the paper we use a new Fourier technique to
obtain a Stein characterizations for random variables in the second
Wiener chaos. We provide the connection between this result and
similar conclusions that can be derived using Malliavin calculus.  We
also introduce a new form of discrepancy which we use, in the second
part of the paper, to provide bounds on the 2-Wasserstein distance
between linear combinations of independent centered random
variables. Our method of proof is entirely original. In particular it
does not rely on estimation of bounds on solutions of the so-called
Stein equations at the heart of Stein's method. We provide several
applications, and discuss comparison with recent similar results on
the same topic.


\tableofcontents

\section{Introduction} 

\subsection{Background}

Stein's method is a popular and versatile probabilistic toolkit for
stochastic approximation.  Presented originally in the context of
Gaussian CLTs with dependant summands (see \cite{S72}) it has now been
extended to cater for a wide variety of quantitative asymptotic
results, see \cite{Chen-book} for a thorough overview of Gaussian
approximation
or~https://sites.google.com/site/steinsmethod for an
up-to-date list of references on non-Gaussian and non-Poisson
Stein-type results.  To this date one of the most active areas of
application of the method is in Gaussian analysis, via Nourdin and
Peccati's so-called Malliavin/Stein calculus on Wiener space, see
\cite{n-pe-1} or Ivan Nourdin's dedicated webpage
https://sites.google.com/site/malliavinstein.

Given two random objects $F, F_{\infty}$, Stein's method allows to compute fine
bounds on quantities of the form
\begin{equation*} \sup_{h\in \mathcal{H}} \left|
    \E\left[ h(F) \right] - \E \left[ h(F_{\infty}) \right] \right|.
\end{equation*}
 The method rests on three pins :  
\begin{enumerate}[A.]
\item a ``Stein pair'', i.e.\ a linear operator and a class of
  functions $(\mathcal{A}_\infty, \mathcal{F}(\mathcal{A}_\infty))$ such that
  $ \E \left[ \mathcal{A}_\infty(f(F_{\infty})) \right] = 0$ for all
  $f \in \mathcal{F}(\mathcal{A}_\infty)$;
\item a contractive inverse operator $\mathcal{A}_\infty^{-1}$ acting on the centered
  functions $\bar{h} = h - \E h(F_{\infty})$ in $\mathcal{H}$ and contraction information,
  i.e.\ tight bounds on $\mathcal{A}_\infty^{-1}(\bar{h})$ and its
  derivatives;
\item handles on the structure of $F$ (such as $F=F_n=T(X_1, \ldots,
  X_n)$ a $U$-statistic, $F = F(X)$ a functional of an isonormal
  Gaussian process, $F$ a statistic on a random graph, etc.). 
\end{enumerate}
Given the conjunction of these three elements one can then apply some
form of transfer principle : 
\begin{equation}\label{eq:6}
 \sup_{h\in \mathcal{H}} \left|
    \E\left[ h(F) \right] - \E \left[ h(F_{\infty}) \right] \right| 
   = \sup_{h\in \mathcal{H}} \left|
    \E\left[ \mathcal{A}_\infty\left( \mathcal{A}_\infty^{-1}(\bar h(F)) \right) \right]\right|;
\end{equation}
remarkably the right-hand-side of the above is often much more
amenable to computations than the left-hand-side, even in particularly
unfavourable circumstances. This has resulted in Stein's method
delivering several striking successes (see
\cite{B-H-J,Chen-book,n-pe-1}) which have led the method to becoming
the recognised and acclaimed tool it is today.

In general the identification of a Stein operator is the cornerstone
of the method. While historically most practical implementations
relied on adhoc arguments, several general tools exist, including
Stein's \emph{density approach} \cite{S86} and Barbour's
\emph{generator approach} \cite{B90}.  A general theory for Stein
operators is available in \cite{LRS}. In many important cases, these  are first
order differential operators  (see \cite{dobler}) or difference
operators (see \cite{LS}). Higher order differential operators have
recently come into light (see  \cite{g-variance-gamma,p-r-r}). 

Once an operator is $\mathcal{A}_{\infty}$ identified, the task is
then to bound the resulting rhs of \eqref{eq:6}; there are many ways
to perform this.  In this paper we focus on Nourdin and Peccati's
approach to the method. Let $F_{\infty}$ be a standard Gaussian random
variable. Then the appropriate operator is
$ \mathcal{A}_{\infty}f(x) = f'(x) - xf(x).$ Given a sufficiently
regular centered random variable $F$ with finite variance and smooth
density, define its Stein kernel $\tau_F(F)$ through the integration
by parts  formula
\begin{equation}\label{eq:9}
  \E[\tau_F(F) \phi'(F) ] = \E \left[ F \phi(F)
\right] \mbox{ for all absolutely continuous } \phi.
\end{equation}
 Then, for $f$ a solution to
$f'(x) - xf(x) = h(x) -\E[h(F_{\infty})]$ write
\begin{align*}
  \E[h(F)] - \E[h(F_{\infty})] &  = \E \left[ f_h'(F)
                                        - F f_h(F) \right]  = \E \left[ (1-\tau_F(F)) f_h'(F) \right]
\end{align*}
so that 
\begin{align*}
  \left|  \E[h(F)] - \E[h(F_{\infty})]  \right| \le \| f_h'\|
  \sqrt{\E \left[ (1-\tau_F(F))^2  \right]}. 
\end{align*} 
At this stage two good things happen : (i) the constant
$\sup_{h \in \mathcal{H}}\|f_h'\|$ (which is intrinsically Gaussian
and does not depend on the law of $F$) is bounded for wide and
relevant classes $\mathcal{H}$; (ii) the quantity
\begin{equation}\label{eq:7}
  S(F \, || \, F_{\infty}) = \E \left[ (1-\tau_F(F))^2 \right]
\end{equation}
(called the \emph{Stein discrepancy}) is tractable, via Malliavin
calculus, as soon as $F$ is a sufficiently regular functional of a
Gaussian process. These two realizations opened a wide field of
applications within the so-called ``Malliavin Stein Fourth moment
theorems'', see \cite{n-pe-ptrf, n-pe-1}.  A similar approach holds
also if $F_{\infty}$ is centered Gamma, see \cite{n-pe-2,a-m-m-p}, and more
generally if the law of the target random variable $F_{\infty}$ belongs to the family of
Variance Gamma distributions, see \cite{thale} for the method and
\cite{g-variance-gamma} for the bounds on the corresponding
solutions. See also \cite{kusuoka,viquez} for other
generalizations. We stress that in the Gaussian case, Stein's method
provides bounds e.g.\ in the Total Variation distance, whereas
technicalities related to the Gamma and Variance Gamma targets impose
that one must deal with smoother distances (i.e.\ integrated
probability measures of the form \eqref{eq:6} with $\mathcal{H}$ a
class of smooth functions) in such cases.

\subsection{Purpose of this paper}

The primary purpose of this paper is to extend Nourdin and Peccati's
``Stein discrepancy analysis'' to provide meaningful bounds on
$ \mathrm{d}(F, F_{\infty})$ for $\mathrm{d}(.,.)$ some appropriate
probability metric and random variables $F_{\infty}$ belonging to the
second Wiener chaos, that is
\begin{equation}\label{target-wiener1}
  F_{\infty} = \sum_{i=1}^q \alpha_{\infty,i} (N^2_i -1).
\end{equation}
where $q\ge 2$, $\{N_i\}_{i=1}^{q}$ are i.i.d. $\mathscr{N}(0,1)$
random variables, and the coefficients
$\{ \alpha_{\infty,i}\}_{i=1}^{q}$ are distinct.

 Such a generalization
immediately runs into a series of obstacles which need to be dealt
with. We single out three crucial questions : (Q1) what operator
$\mathcal{A}_{\infty}$? (Q2) what quantity will play the role of the
Stein discrepancy $S(F \, || \, F_{\infty})$?  (Q3) what kind of
distances $\mathrm{d}(.,.)$ can we tackle through this approach?

In this paper we provide a complete answers to a more general version
of (Q1), hereby opening the way for applications of Stein's method to
a wide variety of new target distributions. We use results from
\cite{a-p-p} to answer (Q2) for chaotic random variables. We also
provide an answer to (Q3) for $\mathrm{d}(.,.)$ the  $p$-Wasserstein
distances with $p\le 2$, under specific assumptions on the structure
of $F$. Such a result extends the scope of Stein's method to so far
unchartered territories, because aside for the case $p=1$,
$p$-Wasserstein distances do not admit a representation of the form
\eqref{eq:6}. 
\subsection{Overview of the results}

In the first part of the paper,
Section~\ref{sec:steins-method-second}, we discuss Stein's method for
target distributions of the form \eqref{target-wiener1}. In
Section~\ref{benjarras} we introduce an entirely new Fourier-based
approach to prove a Stein-type characterization for a large family of
$F_{\infty}$ encompassing those of the form \eqref{target-wiener1}.
The operator $\mathcal{A}_{\infty}$ we obtain is a differential
operator of order $q$.  In Section \ref{sec:mall-based-appr} we use
recent results from \cite{a-p-p} to derive a Malliavin-based
justification for our $\mathcal{A}_{\infty}$ when $F_{\infty}$ is of
the form \eqref{target-wiener1}. We also introduce a new quantity
$\Delta(F_n, F_{\infty})$ for which we will provide a heuristic
justification in Section~\ref{sec:cattyw-steins-meth} of the fact that
$\Delta(F_n, F_{\infty})$ generalizes the Stein discrepancy
$S(F \, || \, F_{\infty})$ in a natural way for chaotic random
variables.  Finally we argue
that 
quantitative assessments for general targets of the form
\eqref{target-wiener1} are out of the scope of the current version of
Stein's method. 

In the second part of the paper, Section \ref{sec:preliminaries}, we
introduce an entirely new \emph{polynomial} approach to Stein's method
 to provide bounds on the Wasserstein-2 distance (and hence the
 Wasserstein-1 distance) in terms of $\Delta(F_n, F_{\infty})$. Our
 approach
bypasses entirely the need for estimating bounds on solutions of Stein
equations.  More specifically we provide a tool for providing
quantitative assessments on $\mathrm{d}_{W_2}(F_n, F_{\infty})$ in terms of the
generalized Stein discrepancy $\Delta(F_n, F_{\infty})$ for
$F_{\infty}$ as in \eqref{target-wiener1} and
\begin{equation*}
F_n = \sum\limits_{i=1}^{\infty} \alpha_{n, i} (N_i^2-1)
\end{equation*}
still with $\left\{ N_i \right\}_{i\ge 1}$ i.i.d.\ standard Gaussian
and now $\left\{ \alpha_{n, i} \right\}_{i\ge1}$ not necessarily distinct
real numbers.   
As mentioned above, the fact that we bound the Wasserstein-2 distance
is not anecdotal : this distance (which is useful in
many important settings, see \cite{villani-book}) does not bear a dual representation
of the form \eqref{eq:6} and is thus entirely out of the scope of the
traditional versions of Stein's method. 
In Section \ref{sec:idea-behind-proof} we provide an intuitive
explanation of the proof of our main results. In Section
\ref{s:applications} we apply our bounds to particular cases and
compare them to the only competitor bounds available in the literature
which are due to \cite{thale}, wherein only the case
\eqref{target-wiener1} with $q=2$ and $\alpha_1 = -\alpha_2$ is covered. Finally
in Section~\ref{lem:appendix} we provide the proof.

\section{Stein's method for the second Wiener
  chaos}\label{sec:steins-method-second}
Here we set up Stein's method for target distributions in the second
Wiener chaos of the form
$$ F_\infty=  \alpha_{\infty,1} (N^2_1 -1) + \cdots +
\alpha_{\infty,q}(N^2_q -1)$$ where $\{N_i\}_{i=1}^{q}$ is a family
of i.i.d.\ $\mathscr{N}(0,1)$ random variables.

\subsection{Overview of known results}\label{sec:overv-known-results}

In the special case when $\alpha_{\infty,i}=1$ for all $i$, then
$F_\infty= \sum_{i=1}^{q} (N^2_i -1) \sim \chi^2_{(q)}$ is a centered
chi-squared random variable with $q$ degree of freedom.  Pickett
\cite{pickett} has shown that a Stein's equation for target
distribution $F_\infty$ is given by the first order differential
equation
$$x f'(x) + \frac12 (q-x) f(x) = h(x) - \E[h(F_\infty)].$$
For more recent results in this direction consult \cite{g-p-r} and
references therein.

Another important contribution in our direction is given by Gaunt in
\cite{g-2normal} with $q=2$ and
$\alpha_{\infty,1}= - \alpha_{\infty,2}= \frac12$. In this case
$$F_\infty \stackrel{\text{law}}{=} N_1 \times N_2 $$ where $N_1$ and $N_2$ are two independent $\mathscr{N} (0,1)$ random variables. He has shown that a Stein's equation for $F_\infty$ can be given by the following second order differential equation

$$x f''(x) + f'(x) - x f(x) = h(x) - \E [h (F_\infty)].$$
It is a well known fact that the density function of the target random
variable $N_1 \times N_2$ is expressible in terms of the modified
Bessel function of the second kind so that it is given by solution of
a known second order differential equation and the Stein operator
follows from some form of duality argument.

The more relevant studies of target distributions having a second
order Stein's differential equations include: Variance-Gamma
distribution \cite{g-variance-gamma}, Laplace distribution
\cite{p-r-laplace}, or a family of probability distributions given by
densities
$$f_s(x)= \Gamma(s) \sqrt{\frac{2}{s \pi}} \exp(- \frac{x^2}{2s})
U(s-1,\frac12,\frac{x^2}{2s}), \quad x>0, \, s \ge \frac12$$ appearing
in preferential attachment random graphs, and $U(a,b;x)$ is the
Kummer's confluent hypergeometric function, see \cite{p-r-r}.
We stress the fact that Stein's method is {completely open} even in
the simple case when the target random variable $F_\infty$ has only
two non-zero eigenvalues $\alpha_{\infty,1}$ and $\alpha_{\infty,2}$,
i.e.
$$F_\infty = \alpha_{\infty,1} (N^2_1 -1) + \alpha_{\infty,2} (N^2_2 -1)$$
such that
$\vert \alpha_{\infty,1}\vert \neq \vert \alpha_{\infty,2} \vert $. It
is worth mentioning that such distributions are beyond the
Variance-Gamma class, and are appearing more and more
in  very recent and delicate limit theorems, see \cite{b-t} for
asymptotic behavior of generalized Rosenblatt process at extreme
critical exponents and \cite{m-p-r-w} for asymptotic nodal
length distributions.

\subsection{Fourier-based  approach}

Before stating the next theorem, we need to introduce some
notations. For any $d$-tuple $(\lambda_1,...,\lambda_d)$ of real
numbers, we define the symmetric elementary polynomial of order
$k\in\{1,...,d\}$ evaluated at $(\lambda_1,...,\lambda_d)$ by:
\begin{align*}
e_{k}(\lambda_1,...,\lambda_d)=\sum_{1\leq i_1<i_2<...<i_k\leq d}\lambda_{i_1}...\lambda_{i_k}.
\end{align*}
We set, by convention, $e_{0}(\lambda_1,...,\lambda_d)=1$. Moreover, for any $(\mu_1,...,\mu_d)\in\mathbb{R}^*$
and any $k\in\{1,...,d\}$, we denote by $(\lambda/\mu)_k$ the $d-1$ tuple defined by:
\begin{align*}
(\frac{\lambda}{\mu})_k=\bigg(\frac{\lambda_1}{\mu_1},...,\frac{\lambda_{k-1}}{\mu_{k-1}},\frac{\lambda_{k+1}}{\mu_{k+1}},...,\frac{\lambda_d}{\mu_d}\bigg).
\end{align*} 
For any $(\alpha,\mu)\in\mathbb{R}_+^*$, we denote by $\gamma(\alpha,\mu)$ a Gamma law with parameters $(\alpha,\mu)$ whose density is:
\begin{align*}
\forall x\in\mathbb{R}_+^*,\ \gamma_{\alpha,\mu}(x)=\dfrac{\mu^\alpha}{\Gamma(\alpha)}x^{\alpha-1}\exp\big(-\mu x\big).
\end{align*}

\begin{thm}\label{benjarras}
Let $d \ge 1$ and $(m_1, \ldots, m_d) \in \mathbb{N}^d$. Let $((\alpha_1,\mu_1),...,(\alpha_d,\mu_d))\in \big(\mathbb{R}_+^*\big)^{2d}$ and $(\lambda_1, \ldots, \lambda_d) \in
\R^{\star}$ and consider: 
\begin{equation*}
   F = -\sum_{i=1}^d\lambda_i\dfrac{m_i\alpha_i}{\mu_i}+\sum_{i=1}^{d}\lambda_i\gamma_i\big(m_i\alpha_i,\mu_i\big),
\end{equation*}
where $\{\gamma_i\big(m_i\alpha_i,\mu_i\big)\}$ is a collection of independent gamma random variables with appropriate parameters.
Let $Y$ be a real valued random variable such that $\E[|Y|]<+\infty$. Then $Y \stackrel{\text{law}}{=} F$ if and only if
\begin{align}
  &\E \bigg[ \big(Y+\sum_{i=1}^d\lambda_i\dfrac{m_i\alpha_i}{\mu_i}\big)(-1)^d\bigg(\prod_{j=1}^d\dfrac{\lambda_j}{\mu_j}\bigg)\phi^{(d)}(Y)+\sum_{l=1}^{d-1}(-1)^l\bigg(Ye_{l}(\frac{\lambda_1}{\mu_1},...,\frac{\lambda_d}{\mu_d})\nonumber\\ &+\sum_{k=1}^d\lambda_k\dfrac{m_k\alpha_k}{\mu_k}\left(e_{l}(\frac{\lambda_1}{\mu_1},...,\frac{\lambda_d}{\mu_d})-e_{l}((\frac{\lambda}{\mu})_k)\right)\bigg)\phi^{(l)}(Y)+Y\phi(Y) \bigg]=0,\label{eq:2}
\end{align}
for all $\phi\in S(\mathbb{R})$.
\end{thm}
\begin{proof}
$(\Rightarrow)$. Let $F$ be as in the statement of the theorem. We denote by $J_+=\{j\in\{1,...,d\}:\lambda_j>0\}$ and $J_-=\{j\in\{1,...,d\}:\lambda_j<0\}$. Let us compute the characteristic function of $F$. For any $\xi\in \mathbb{R}$, we have:
\begin{align*}
\phi_F(\xi)&=\E[\exp(i\xi F)],\\
&=\exp\bigg(-i\xi\sum_{k=1}^d\lambda_k\dfrac{m_k\alpha_k}{\mu_k}\bigg)\prod_{j=1}^d\E\bigg[\exp\bigg(i\xi\lambda_j\gamma_j(m_j\alpha_j,\mu_j)\bigg)\bigg],\\
&=\exp\bigg(-i\xi<m\alpha;\lambda/\mu>\bigg)\prod_{j\in J_+}\bigg(\exp\bigg(\int_0^{+\infty}\bigg(e^{i\xi\lambda_j x}-1\bigg)\bigg(\frac{m_j\alpha_j}{x}e^{-\mu_jx}\bigg)dx\bigg)\bigg)\\
&\times\prod_{j\in J_-}\bigg(\exp\bigg(\int_0^{+\infty}\bigg(e^{i\xi\lambda_j x}-1\bigg)\bigg(\frac{m_j\alpha_j}{x}e^{-x\mu_j}\bigg)dx\bigg)\bigg),\\
&=\exp\bigg(-i\xi<m\alpha;\lambda/\mu>\bigg)\exp\bigg(\int_0^{+\infty}\bigg(e^{i\xi x}-1\bigg)\bigg(\sum_{j\in J_+}\frac{m_j\alpha_j}{x}e^{-\frac{x\mu_j}{\lambda_j}}\bigg)dx\bigg)\\
&\times\exp\bigg(\int_0^{+\infty}\bigg(e^{-i\xi x}-1\bigg)\bigg(\sum_{j\in J_-}\frac{m_j\alpha_j}{x}e^{-\frac{x\mu_j}{(-\lambda_j)}}\bigg)dx\bigg),
\end{align*} 
where we have used the Lévy-Khintchine representation of the Gamma distribution. We denote $\mu_j/\lambda_j$ by $\nu_j$. Differentiating with respect to $\xi$ together with standard computations, we obtain:
\begin{align*}
\prod_{k=1}^d(\nu_k-i\xi)\dfrac{d}{d\xi}\bigg(\phi_F(\xi)\bigg)=\bigg[-i<m\alpha,\lambda/\mu>\prod_{k=1}^d(\nu_k-i\xi)+i\sum_{k=1}^dm_k\alpha_k\prod_{l=1,l\ne k}^d(\nu_l-i\xi)\bigg]\phi_F(\xi).
\end{align*}
Let us introduce two differential operators characterized by their symbols in Fourier domain. For smooth enough test functions, $\phi$, we define:
\begin{align*}
&\mathcal{A}_{d,\nu}(\phi)(x)=\frac{1}{2\pi}\int_{\mathbb{R}}\mathcal{F}(\phi)(\xi)\bigg(\prod_{k=1}^d(\nu_k-i\xi)\bigg)\exp(ix\xi)d\xi,\\
&\mathcal{B}_{d,m,\nu}(\phi)(x)=\frac{1}{2\pi}\int_{\mathbb{R}}\mathcal{F}(\phi)(\xi)\bigg(\sum_{k=1}^dm_k\alpha_k\prod_{l=1,l\ne k}^d(\nu_l-i\xi)\bigg)\exp(ix\xi)d\xi,\\
&\mathcal{F}(\phi)(\xi)=\int_{\mathbb{R}}\phi(x)\exp(-ix\xi)dx.
\end{align*}
Integrating against smooth test functions the differential equation satistifed by the characteristic function $\phi_F$, we have, for the left hand side:
\begin{align*}
\int_{\mathbb{R}}\mathcal{F}(\phi)(\xi)\bigg(\prod_{k=1}^d(\nu_k-i\xi)\bigg)\dfrac{d}{d\xi}\bigg(\phi_F(\xi)\bigg)d\xi&=\int_{\mathbb{R}}\mathcal{F}\big(\mathcal{A}_{d,\nu}(\phi)\big)(\xi)\dfrac{d}{d\xi}\bigg(\phi_F(\xi)\bigg)d\xi,\\
&=-\int_{\mathbb{R}}\dfrac{d}{d\xi}\bigg(\mathcal{F}\big(\mathcal{A}_{d,\nu}(\phi)\big)(\xi)\bigg)\phi_F(\xi)d\xi,\\
&=i\int_{\mathbb{R}}\mathcal{F}\big(x\mathcal{A}_{d,\nu}(\phi)\big)(\xi)\phi_F(\xi)d\xi,
\end{align*}
where we have used the standard fact $d/d\xi(\mathcal{F}(f)(\xi))=-i\mathcal{F}(xf)(\xi)$. Similarly, for the right hand side, we obtain:
\begin{align*}
\operatorname{RHS}&=\int_{\mathbb{R}}\mathcal{F}(\phi)(\xi)\bigg[-i<m\alpha,\lambda/\mu>\prod_{k=1}^d(\nu_k-i\xi)+i\sum_{k=1}^dm_k\alpha_k\prod_{l=1,l\ne k}^d(\nu_l-i\xi)\bigg]\phi_F(\xi)d\xi,\\
&=i\int_{\mathbb{R}}\mathcal{F}\big(-<m\alpha,\lambda/\mu>\mathcal{A}_{d,\nu}(\phi)+\mathcal{B}_{d,m,\nu}(\phi)\big)(\xi)\phi_F(\xi)d\xi.
\end{align*}
Thus,
\begin{align*}
\int_{\mathbb{R}}\mathcal{F}\big((x+<m\alpha,\lambda/\mu>)\mathcal{A}_{d,\nu}(\phi)-\mathcal{B}_{d,m,\nu}(\phi)\big)(\xi)\phi_F(\xi)d\xi=0
\end{align*}
Going back in the space domain, we obtain the following Stein-type characterization formula:
\begin{align*}
\E\big[(F+<m\alpha,\lambda/\mu>)\mathcal{A}_{d,\nu}(\phi)(F)-\mathcal{B}_{d,m,\nu}(\phi)(F)\big]=0.
\end{align*}
In order to conclude the first half of the proof, we need to compute explicitely the coefficients of the operators $\mathcal{A}_{d,\nu}$ and $\mathcal{B}_{d,m,\nu}$ in the following expansions:
\begin{align*}
\mathcal{A}_{d,\nu}=\sum_{k=0}^da_k\dfrac{d^k}{dx^k},\\
\mathcal{B}_{d,m,\nu}=\sum_{k=0}^{d-1}b_k\dfrac{d^k}{dx^k}.
\end{align*}
First of all, let us consider the following polynomial in $\mathbb{R}[X]$:
\begin{align*}
P(x)=\prod_{j=1}^d(\nu_j-x)=(-1)^d\prod_{j=1}^d(x-\nu_j).
\end{align*}
We denote by $p_0,...,p_d$ the coefficients of $\prod_{j=1}^d(X-\nu_j)$ in the basis $\{1,X,...,X^d\}$. Vieta formula readily give:
\begin{align*}
\forall k\in\{0,...,d\},\ p_k=(-1)^{d+k}e_{d-k}(\nu_1,...,\nu_d),
\end{align*}
It follows that the Fourier symbol of $\mathcal{A}_{d,\nu}$ is given by:
\begin{align*}
\prod_{k=1}^d(\nu_k-i\xi)=P(i\xi)=\sum_{k=0}^d(-1)^ke_{d-k}(\nu_1,...\nu_d)(i\xi)^k.
\end{align*}
Thus, we have, for $\phi$ smooth enough:
\begin{align*}
\mathcal{A}_{d,\nu}(\phi)(x)=\sum_{k=0}^d(-1)^{k}e_{d-k}(\nu_1,...,\nu_d)\phi^{(k)}(x).
\end{align*}
Let us proceed similarly for the operator $B_{d,m,\nu}$. We denote by $P_k$ the following polynomial in $\mathbb{R}[X]$ (for any $k\in\{1,...,d\}$):
\begin{align*}
P_k(x)=(-1)^{d-1}\prod_{l=1,l\ne k}^d(x-\nu_l).
\end{align*}
A similar argument provides the following expression:
\begin{align*}
P_k(x)=\sum_{l=0}^{d-1}(-1)^le_{d-1-l}(\underline{\nu}_k)x^l,
\end{align*}
where $\underline{\nu}_k=(\nu_1,...,\nu_{k-1},\nu_{k+1},...,\nu_d)$. Thus, the symbol of the differential operator $B_{d,m,\nu}$ is given by:
\begin{align*}
\sum_{k=1}^dm_k\alpha_k\prod_{l=1,l\ne k}^d(\nu_l-i\xi)=\sum_{l=0}^{d-1}(-1)^l\bigg(\sum_{k=1}^dm_k\alpha_ke_{d-1-l}(\underline{\nu}_k)\bigg)(i\xi)^l.
\end{align*}
Thus, we have:
\begin{align*}
B_{d,m,\nu}(\phi)(x)=\sum_{l=0}^{d-1}(-1)^l\bigg(\sum_{k=1}^dm_k\alpha_ke_{d-1-l}(\underline{\nu}_k)\bigg)\phi^{(k)}(x).
\end{align*}
Consequently, we obtain:
\begin{align*}
&\E\big[(F+<m\alpha,\lambda/\mu>)\sum_{k=0}^d(-1)^{k}e_{d-k}(\nu_1,...,\nu_d)\phi^{(k)}(F)\\
&-\sum_{l=0}^{d-1}(-1)^l\bigg(\sum_{k=1}^dm_k\alpha_ke_{d-1-l}(\underline{\nu}_k)\bigg)\phi^{(k)}(F)\big]=0.
\end{align*}
Finally, there is a straightforward relationship between $e_{k}(\nu_1,...,\nu_d)$ and $e_{d-k}(\lambda_1/\mu_1,...,\lambda_d/\mu_d)$. Namely,
\begin{align*}
\forall k\in\{0,...,d\},\ e_{k}(\nu_1,...,\nu_d)=\dfrac{\prod_{j=1}^d\mu_j}{\prod_{j=1}^d\lambda_j}e_{d-k}(\frac{\lambda_1}{\mu_1},...,\frac{\lambda_d}{\mu_d}).
\end{align*}
Thus, multiplying by $\prod_{j=1}^d\lambda_j/\prod_{j=1}^d\mu_j$, the previous Stein-type characterisation equation, we have:
\begin{align*}
&\E\bigg[(F+<m\alpha,\lambda/\mu>)(-1)^d\bigg(\prod_{j=1}^d\frac{\lambda_j}{\mu_j}\bigg)\phi^{(d)}(F)+\sum_{l=1}^{d-1}(-1)^l\bigg(Fe_{l}(\frac{\lambda_1}{\mu_1},...,\frac{\lambda_d}{\mu_d})\\
&+\sum_{k=1}^d\lambda_km_k\frac{\alpha_k}{\mu_k}\big(e_{l}(\frac{\lambda_1}{\mu_1},...,\frac{\lambda_d}{\mu_d})-e_{l}((\frac{\lambda}{\mu})_k)\big)\bigg)\phi^{(l)}(F)+F\phi(F)\bigg]=0.
\end{align*}
$(\Leftarrow)$ Let $Y$ be a real valued random variable such that $\E[|Y|]<+\infty$ and:
\begin{align*}
\forall \phi\in S(\mathbb{R}),\ &\E \bigg[ (Y+<m\alpha,\lambda/\mu>)(-1)^d\bigg(\prod_{j=1}^d\frac{\lambda_j}{\mu_j}\bigg)\phi^{(d)}(Y)+\sum_{l=1}^{d-1}(-1)^l\bigg(Ye_{l}(\frac{\lambda_1}{\mu_1},...,\frac{\lambda_d}{\mu_d})\\ &+\sum_{k=1}^d\lambda_km_k\frac{\alpha_k}{\mu_k}\big(e_{l}(\frac{\lambda_1}{\mu_1},...,\frac{\lambda_d}{\mu_d})-e_{l}((\frac{\lambda}{\mu})_k)\big)\bigg)\phi^{(l)}(Y)+Y\phi(Y) \bigg]=0.
\end{align*}
By the previous step, this implies that:
\begin{align*}
&\forall \phi\in S(\mathbb{R}),\ \int_{\mathbb{R}}\mathcal{F}\big((x+<m\alpha,\lambda/\mu>)\mathcal{A}_{d,\nu}(\phi)-\mathcal{B}_{d,m,\nu}(\phi)\big)(\xi)\phi_Y(\xi)d\xi=0,\\
&\Leftrightarrow\ \int_{\mathbb{R}}\mathcal{F}\big(x\mathcal{A}_{d,\nu}(\phi)\big)(\xi)\phi_Y(\xi)d\xi=\int_{\mathbb{R}}\mathcal{F}\big(-<m\alpha,\lambda/\mu>\mathcal{A}_{d,\nu}(\phi)+\mathcal{B}_{d,m,\nu}(\phi)\big)(\xi)\phi_Y(\xi)d\xi,\\
&\Leftrightarrow\ \prod_{k=1}^d(\nu_k-i\xi)\dfrac{d}{d\xi}\bigg(\phi_Y\bigg)(.)=\bigg[-i<m\alpha,\lambda/\mu>\prod_{k=1}^d(\nu_k-i\xi)+i\sum_{k=1}^d\alpha_km_k\prod_{l=1,l\ne k}^d(\nu_l-i\xi)\bigg]\phi_Y(.),
\end{align*}
in $S'(\mathbb{R})$. Since $\E[|Y|]<+\infty$, the characteristic function of $Y$ is differentiable on the whole real line so that:
\begin{align*}
\forall\xi\in\mathbb{R},\ \dfrac{d}{d\xi}\bigg(\phi_Y\bigg)(\xi)=\bigg[-i<m\alpha,\lambda/\mu>+i\sum_{k=1}^dm_k\alpha_k\dfrac{1}{\nu_k-i\xi}\bigg]\phi_Y(\xi)
\end{align*}
Moreover, we have $\phi_Y(0)=1$. Thus, by Cauchy-Lipschitz theorem, we have:
\begin{align*}
\forall\xi\in\mathbb{R},\ \phi_Y(\xi)=\phi_F(\xi).
\end{align*}
This concludes the proof of the theorem.
\end{proof}

Taking $\alpha_k=\mu_k=1/2$ in the previous theorem implies the following straightforward corollary:

\begin{cor}\label{benjarras2}
Let $d \ge 1$, $q\geq 1$ and $(m_1, \ldots, m_d) \in \mathbb{N}^d$ such that $m_1+...+m_d=q$. Let $(\lambda_1, \ldots, \lambda_d) \in
\R^{\star}$ and consider: 
\begin{equation*}
   F = \sum_{i=1}^{m_1}\lambda_1(N_i^2-1)+\sum_{i=m_1+1}^{m_1+m_2}\lambda_2(N_i^2-1)+...+\sum_{i=m_1+\ldots+m_{d-1}+1}^q\lambda_d(N_i^2-1),
\end{equation*}
Let $Y$ be a real valued random variable such that $\E[|Y|]<+\infty$. Then $Y \stackrel{\text{law}}{=} F$ if and only if
\begin{align}
  &\E \bigg[ \big(Y+\sum_{i=1}^d\lambda_im_i\big)(-1)^d2^d\bigg(\prod_{j=1}^d\lambda_j\bigg)\phi^{(d)}(Y)+\sum_{l=1}^{d-1}2^l(-1)^l\bigg(Ye_{l}(\lambda_1,...,\lambda_d)\nonumber\\ &+\sum_{k=1}^d\lambda_km_k\left(e_{l}(\lambda_1,...,\lambda_d)-e_{l}((\underline{\lambda}_k)\right)\bigg)\phi^{(l)}(Y)+Y\phi(Y) \bigg]=0,\label{eq:3}
\end{align}
for all $\phi\in S(\mathbb{R})$.
\end{cor}
\begin{proof} Let $F$ be as in the statement of the theorem. Then, it is sufficient to observe that we have the following equality in law:
\begin{align*}
F\stackrel{\text{law}}{=}-\sum_{k=1}^dm_k\lambda_k+\sum_{i=1}^d\lambda_i\gamma_i\big(\frac{m_i}{2},\frac{1}{2}\big).
\end{align*}
To end the proof of the corollary, we apply the previous theorem with $\alpha_k=\mu_k=1/2$ for every $k$.
\end{proof}

\begin{ex}\label{ex:eicheltha}{ \rm 
Let $d=1$, $m_1=q\geq 1$ and $\lambda_1=\lambda>0$. The differential operator reduces to (on smooth test function $\phi$):
\begin{align*}
-2\lambda(x+q\lambda)\phi^{(1)}(x)+x\phi(x).
\end{align*}
This differential operator is similar to the one characterising the gamma distribution of parameters $(q/2,1/(2\lambda))$. Indeed, we have, for $F\stackrel{\text{law}}{=}\gamma\big(q/2,1/(2\lambda)\big)$, on smooth test function, $\phi$:
\begin{align*}
\E\bigg[F\phi^{(1)}\big(F\big)+\big(\frac{q}{2}-\frac{F}{2\lambda}\big)\phi\big(F\big)\bigg]=0
\end{align*}
We can move from the first differential operator to the second one by performing a scaling of parameter $-1/(2\lambda)$ and the change of variable $x=y-q\lambda$.
}
\end{ex}

\begin{ex}{ \rm 
Let $d=2,\ q=2$, $\lambda_1=-\lambda_2=1/2$ and $m_1=m_2=1$. The differential operator reduces to (on smooth test function $\phi$):
\begin{align*}
T(\phi)(x)&=4(x+<m,\lambda>)\lambda_1\lambda_2\phi^{(2)}(x)-2\big[xe_1(\lambda_1,\lambda_2)+\lambda_1m_1(e_1(\lambda_1,\lambda_2)-e_1(\lambda_2))\\
&+\lambda_2m_2(e_1(\lambda_1,\lambda_2)-e_1(\lambda_1))\big]\phi^{(1)}(x)+x\phi(x),\\
&=-x\phi^{(2)}(x)-\phi^{(1)}(x)+x\phi(x),
\end{align*}
where we have used the fact that $e_1(\lambda_1,\lambda_2)=\lambda_1+\lambda_2=0,\ e_1(\lambda_2)=\lambda_2=-1/2,\ e_1(\lambda_1)=\lambda_1=1/2$. Therefore, up to a minus sign factor, we retrieve the differential operator associated with the random variable:
\begin{align*}
F=N_1\times N_2.
\end{align*}
}
\end{ex}

\subsection{Malliavin-based approach}\label{sec:mall-based-appr}
In this section, we assume that the random objects we consider do live
in the Wiener space. Let $\rm X=\{X(h); \ h \in \HH\}$ stand for an
isonormal process over a separable Hilbert space $\HH$. The reader may
consult \cite[Chapter 2]{n-pe-1} for a detailed discussion on this
topic. The main aim of this section is to use Malliavin calculus on
the Wiener space to obtain a Stein characterization for target random
variables of the form $(\ref{target-wiener1})$. 
The following definition includes the iterated Malliavin
$\Gamma$-operators that lie at the core of this approach. The notation
$\mathbb{D}^{\infty}$ stands for the class of infinitely many times
Malliavin differentiable random variables.

\begin{definition}[see \cite{n-pe-1}]\label{Def : Gamma}  Let $F\in \mathbb{D}^{\infty}$. The sequence of random variables $\{\Gamma_i(F)\}_{i\geq 0}\subset
\mathbb{D}^\infty$ is recursively defined as follows. Set $\Gamma_0(F) = F$
and, for every $i\geq 1$, \[\Gamma_{i}(F) = \langle DF,-DL^{-1}\Gamma_{i-1}(F)\rangle_{\HH}.
\]
For instance, one has that $\Gamma_1(F) = \langle
DF,-DL^{-1}F\rangle_{\HH}= \tau_F(F)$ the Stein kernel of $F$.
\end{definition}
For further use, we also recall that (see again \cite{n-pe-1}) the
cumulants of the random element $F$ and the iterated Malliavin
$\Gamma$- operators are linked by the relation
$$\kappa_{r+1}(F)=r! \E [\Gamma_r(F)] \mbox{ for } r=0,1,\cdots.$$
Following \cite{n-po-1,a-p-p}, we define two crucial polynomials $P$
and $Q$ as follows:
\begin{equation}\label{polynomialP}
 Q(x)=\big( P(x)\big)^{2}=\Big(x \prod_{i=1}^{q}(x - \alpha_{\infty,i} ) \Big)^{2}.
\end{equation}
Finally, for any random element $F$, we define the following quantity
(whose first appearance is in \cite{a-p-p})
\begin{equation}\label{eq:Delta}
\Delta(F, F_{\infty}):= \sum_{r=2}^{\text{deg}(Q)} \frac{Q^{(r)}(0)}{r!} \frac{\kappa_{r}(F)}{2^{r-1}(r-1)!}.
\end{equation}

\begin{prop}\cite[Proposition 3.2]{a-p-p}\label{p:static}
Let $F$ be a centered random variable living in a finite sum of Wiener chaoses. Moreover, assume that
\begin{itemize}
 \item[\bf (i)]$\kappa_r (F) = \kappa_r (F_\infty)$, for all $2 \le r \le k+1=\text{deg}(P)$, and
\item[\bf (ii)] 

\begin{equation*}
              \E \Bigg[ \sum_{r=1}^{k+1} \frac{P^{(r)}(0)}{r! \
                2^{r-1}} \Big( \Gamma_{r-1}(F) -   \E[\Gamma_{r-1}(F)]
              \Big) \Bigg]^2 = 0. 
  \end{equation*}
\end{itemize}
Then, $F \stackrel{{\rm law}}{=} F_\infty,$ and $F$ belongs to the second Wiener chaos.
\end{prop}



In fact item {\bf (ii)} of Proposition \ref{p:static} can be used to
derive a Stein equation for $F_\infty$. To this end, set
\begin{align*}
 &  a_l= \frac{P^{(l)}(0)}{l! 2^{l-1}} \quad 1 \le l \le q+1,\\
&  b_l= \sum_{r=l}^{q+1} a_r \E[\Gamma_{r-l+1}(F_\infty)] =  \sum_{r=l}^{q+1} \frac{a_r}{(r-l+1)!} \kappa_{r-l+2}(F_\infty)  \quad 2 \le l \le q+1
\end{align*}
Now, we introduce the following differential operator of order $q$
(acting on functions $f \in C^q(\R)$) : 
\begin{equation}\label{eq:SME} 
\mathcal{A}_\infty f (x):= \sum_{l=2}^{q+1} (b_l - a_{l-1} x ) f^{(q+2-l)}(x) - a_{q+1} x f(x).
\end{equation}
Then, we have the following result.

\begin{thm}\label{thm:SMC}
Assume that $F$ is a general centered random variable living in a finite sum of Wiener chaoses (and hence smooth in the
sense of Malliavin calculus).  Then  $$F \stackrel{\text{law}}{=}
F_\infty$$ if and only if $\E \left[ \mathcal{A}_\infty f (F) \right]
=0$ for all mappings $f:\R \to \R$ such that $\E \left[ \vert
\mathcal{A}_\infty f (F) \vert \right] < \infty$, and moreover
$\E[f^{(q)}(F)^2]< +\infty$. 
\end{thm} 
\begin{proof}
  Repeatedly using the Malliavin integration by parts formulae
  \cite[Theorem 2.9.1]{n-pe-1}, we obtain for any $2 \le l \le q+2$
  that
\begin{align}
  \E\left[F f^{(q-l+2)}(F) \right]
  =   \E\left[ f^{(q)}(F)
  \Gamma_{l-2}(F) \right] +  \sum_{r=q-l+3}^{q-1} \E\left[
  f^{(r)}(F) \right] \E\left[ \Gamma_{r+l-q-2}(F) \right].\label{eq:com1}
\end{align}
For indices $l=2,3$, the second term in the right hand side of
$(\ref{eq:com1})$ is understood to be $0$. Summing from $l=2$ up to
$l=q+2$, we obtain that

\begin{equation}\label{com2}
\begin{split}
\sum_{l=2}^{q+2} a_{l-1} \, \E\left[F f^{(q-l+2)}(F) \right]  & = \sum_{l=2}^{q+2} a_{l-1}  \, \E\left[ f^{(q)}(F) \Gamma_{l-2}(F) \right] \\
& \hskip1cm +  \sum_{l=4}^{q+2} a_{l-1}  \, \sum_{r=q-l+3}^{q-1} \E\left[ f^{(r)}(F) \right] \E\left[ \Gamma_{r+l-q-2}(F) \right]\\
& =  \sum_{l=1}^{q+1} a_{l}  \, \E\left[ f^{(q)}(F) \Gamma_{l-1}(F) \right] \\
& \hskip1cm +  \sum_{l=3}^{q+1} a_{l}  \, \sum_{r=q-l+2}^{q-2} \E\left[ f^{(r)}(F) \right] \E\left[ \Gamma_{r+l-q-1}(F) \right]\\
& = \sum_{l=1}^{q+1} a_{l}  \, \E\left[ f^{(q)}(F) \Gamma_{l-1}(F) \right] \\
& \hskip1cm +  \sum_{l=2}^{q+1} a_{l}  \, \sum_{r=1}^{l-2} \E \left[ f^{(q-r)}(F) \right] \E \left[ \Gamma_{l-r-1}(F) \right].
\end{split}
\end{equation}
On the other hand, 
\begin{equation}\label{eq:com4}
\begin{split}
\sum_{l=2}^{q+1} b_l \, \E \left[ f^{(q+2-l)}(F) \right] &=
\sum_{l=0}^{q-1} b_{l+2}  \E \left[ f^{(q-l)} (F) \right]\\
&=   \sum_{l=0}^{q-1}  \left[ \sum_{r=l+2}^{q+1} a_r \E (
  \Gamma_{r-l-1}(F_\infty) ) \right] \E \left[ f^{(q-l)}(F) \right]\\ 
&= \sum_{r=2}^{q+1} a_r \sum_{l=0}^{r-2} \E \left[
  \Gamma_{r-l-1}(F_\infty)  \right] \times \E \left[ f^{(q-l)}(F)
\right]. 
\end{split}
\end{equation}
Wrapping up, we finally  arrive at
\begin{equation}\label{eq:com5}
\begin{split}
\E \left[ \mathcal{A}_\infty f (F) \right] & = - \E \Bigg[  f^{(q)}(F) \times \Big( \sum_{r=1}^{q+1} a_r \left[ \Gamma_{r-1}(F) - \E[\Gamma_{r-1}(F)] \right] \Big) \Bigg] \\
& \hskip 1cm +  \sum_{r=2}^{q+1} a_r \sum_{l=0}^{r-2}  \left\{ \E [
  f^{(q-l)}(F) ] \times \Big(  \E \left[ \Gamma_{r-l-1}(F_\infty)
  \right] - \E \left[ \Gamma_{r-l-1}(F) \right] \Big) \right\}\\ 
& = - \E \Bigg[  f^{(q)}(F) \times \Big( \sum_{r=1}^{q+1} a_r \left[
  \Gamma_{r-1}(F) - \E[\Gamma_{r-1}(F)] \right] \Big) \Bigg] \\ 
&\hskip 1cm +  \sum_{r=2}^{q+1} a_r \sum_{l=0}^{r-2} \frac{ \E [ f^{(q-l)}(F) ]}{(r-l-1)!} \times \Big( \kappa_{r-l}(F_\infty) - \kappa_{r-l}(F) \Big).
\end{split}
\end{equation}
We are now in a position to prove the claim.  First we assume that
$F \stackrel{\text{law}}{=} F_\infty$. Then obviously
$\kappa_{r}(F)=\kappa_{r}(F_\infty)$ for $r=2,\cdots,2q+2$. Following
the same arguments as in the proof of \cite[Proposition 3.2]{a-p-p},
one can infer that, in fact, $F$ belongs to the second Wiener
chaos. Hence, according to \cite[Lemma 3.1]{a-p-p}, and the
Cauchy--Schwarz inequality, we obtain that
\begin{equation*}
\begin{split}
\vert \E \left[ \mathcal{A}_\infty f (F) \right] \vert & \le \sqrt{\E \left[f^{(q)}(F) \right]^2} \times 
\sqrt{ \E \Big[ \sum_{r=1}^{q+1} a_r \left( \Gamma_{r-1}(F) -
    \E[\Gamma_{r-1}(F)] \right) \Big]^2} \\
& =  \sqrt{\E \left[f^{(q)}(F) \right]^2} \times \sqrt{\Delta(F,F_\infty)}\\
&=  \sqrt{\E \left[f^{(q)}(F) \right]^2} \times \sqrt{\Delta(F_\infty,F_\infty)}= 0.
\end{split}
\end{equation*}
Conversely, assume that $\E \left[ \mathcal{A}_\infty f (F) \right] =0$ for all the suitable functions $f$. Then relation $(\ref{eq:com5})$ implies that, by choosing appropriate 
polynomials for function $f$, we have $\kappa_r(F)=\kappa_r(F_\infty)$ for $r=2,\cdots,q+1$. Now, combining this observation together with relation $(\ref{eq:com5})$, we infer that 
$$\E \left[ \sum_{r=1}^{q+1} a_r \Big( \Gamma_{r-1}(F) - \E[\Gamma_{r-1}(F)] \Big) \Big \vert F \right]=0.$$
Using e.g. integrations by parts, the latter equation can be turned into a linear recurrent relation between the cumulants of $F$ of order up to $q+1$. Combining this with the knowledge of the 
$q+1$ first cumulants characterize all the cumulants of $F$ and hence the distribution $F$. Indeed, all the distributions in the second Wiener chaos are determined by their moments.
\end{proof}

 \begin{ex}{ \rm Consider the special case of only two non-zero distinct eigenvalues $\lambda_1$ and $\lambda_2$, i.e.
\begin{equation}\label{eq:target}
F_\infty=\lambda_1(N^2_1 -1) + \lambda_2 (N^2_2 -1)
\end{equation}
where $N_1, N_2 \sim \mathscr N (0,1)$ are independent. In this case, the polynomial $P$ takes the form $P(x)=x (x - \alpha)(x - \beta)$. Simple 
calculations reveal that $P'(0)=\lambda_1 \lambda_2, P''(0)= -2 (\lambda_1 + \lambda_2)$, and $P^{(3)}(0)=3!$. Also, $\kappa_2(F_\infty)=\E [\Gamma_{1}(F_\infty)]=2(\lambda_1^2 + \lambda_2^2)$, and 
$\kappa_3(F_\infty)= 2 \E [\Gamma_{2}(F_\infty)]= 4 (\lambda_1^3 + \lambda_2^3)$. Then, the Stein equation $(\ref{eq:SME})$ reduces to that  
\begin{equation}\label{eq:stein-equation-2}
 \mathcal{A}_\infty f (x) = -4 (\lambda_1 \lambda_2 x + (\lambda_1 + \lambda_2) \lambda_1 \lambda_2) f''(x) + 2 \left(\lambda_1^2 + \lambda_2^2+ (\lambda_1 + \lambda_2) x \right) f'(x) - x f(x) 
 \end{equation}
 We also remark that when $\lambda_1=-\lambda_2=\frac12$, and hence
 $F_\infty \stackrel{\text{law}}{=} N_1 \times N_2$, the Stein's
 equation $(\ref{eq:stein-equation-2})$ coincides with that in
 \cite[equation (1.9)]{g-2normal}.}
 \end{ex}

 One can show that \eqref{eq:SME} is a specification of \eqref{eq:3}
 in the particular setting considered in the current Section. The
 proof of this asserstion is quite technical and we do not include it.
 To convince the reader, let us inspect the particular case where
 $d=2$, $q=2$ and $\lambda_1\ne\lambda_2$. Using the previous
 corollary, the differential operator, $T$, boils down to, (on smooth
 test function $\phi$):
\begin{align*}
T(\phi)(x)=4\lambda_1\lambda_2(x+\lambda_1+\lambda_2)\phi^{(2)}(x)-2\big[x(\lambda_1+\lambda_2)+\lambda_1^2+\lambda_2^2\big]\phi^{(1)}(x)+x\phi(x),
\end{align*} 
which is, up to a minus sign factor, the differential operator
appearing in Equation \ref{eq:stein-equation-2} with
$\alpha=\lambda_1$ and $\beta=\lambda_2$. In the case of $d$ different
eigenvalues whose respective multiplicities are equal to $1$, in order
to switch from one operator to the other one, we must use the
fundamental Newton-Girard identities linking the elementary symmetric
polynomials valued at $(\lambda_1,...,\lambda_d)$ together with the
power sums valued at $(\lambda_1,...,\lambda_d)$ (which reduce, up to
a multiplicative factor, to the cumulants of $F$).

\subsection{Cattywampus Stein's method}
\label{sec:cattyw-steins-meth}

In this section, we first propose a heuristic to clarify why
$\Delta(F, F_{\infty})$ can be interpreted as a version of the
Malliavin-Stein discrepancy, when $F$ belongs to the second Wiener
chaos. Let $\mathcal{H}$ be a family of test functions that at least
characterizes the convergence in distribution, in the sense that
$$\ud_{\mathcal{H}} (F,F_\infty): = \sup_{h \in \mathcal{H}} \Big \vert   \E [h(F)] - \E[h(F_\infty)]  \Big \vert \approx 0 \quad   \text{if and only if} \quad F \stackrel{\text{law}}{\approx} F_\infty.$$
Let $h \in \mathcal{H}$ and take $\mathcal{A}_\infty$ as in
$(\ref{eq:SME})$. Consider the  Stein
equation
\begin{equation}\label{eq:NHSE}
\mathcal{A}_{\infty}f (x) = h (x) - \E[h(F_{\infty})] \qquad h \in \mathcal{H}.
\end{equation}

\begin{assu}{\bf [Stein universality assumption]} \label{assu:1} For
  any $h \in \mathcal{H}$, equation $(\ref{eq:NHSE})$ has a unique bounded
  solution $f_h$ so that
  \begin{equation}
    \label{eq:4}
\Vert f^{(r)}_h \Vert_\infty < \infty \qquad \forall \, r=1,\cdots,q
  \end{equation}
uniformly in $h$.
\end{assu}

Now, we take a general centered random variable $F$, smooth in the
sense of Malliavin differentiability, for example $F$ belongs to a
finite sum of Wiener chaoses. Then, under Assumption $\ref{assu:1}$,
the proof of Theorem \ref{thm:SMC} reveals that for some general
constants $C_1,C_2$, using the Cauchy-Swartz inequality, one has

\begin{equation}\label{eq:8}
\begin{split}
\sup_{h \in \mathcal{H}} \Big \vert \E [h(F)] - \E[h(F_\infty)] \Big \vert & \le  C_1 \E \Big \vert \sum_{r=1}^{q+1} a_r \left( \Gamma_{r-1}(F) - \E[\Gamma_{r-1}(F)] \right) \Big\vert \\
& \hskip 1cm + C_2 \sum_{r=2}^{q+1} \big \vert \kappa_r(F) - \kappa_r(F_\infty) \big \vert  \\
& \le \sqrt{ \E \Big[ \sum_{r=1}^{q+1} a_r \left( \Gamma_{r-1}(F) - \E[\Gamma_{r-1}(F)] \right) \Big]^2 } \\
 & \hskip 1cm + C_2 \sum_{r=2}^{q+1} \big \vert \kappa_r(F) - \kappa_r(F_\infty) \big \vert.
\end{split}
\end{equation}
However, taking into account \cite[Lemma 3.1]{a-p-p} when $F$ itself
belongs to the second Wiener chaos, then  
\begin{equation}\label{eq:11}
\begin{split}
\sup_{h \in \mathcal{H}} \Big \vert \E [h(F)] - \E[h(F_\infty)] \Big
\vert  \le C_1 \sqrt{\Delta(F,F_\infty)} + C_2 \sum_{r=2}^{q+1} \big
\vert \kappa_r(F) - \kappa_r(F_\infty) \big \vert. 
\end{split}
\end{equation}
Starting from \eqref{eq:11}, to apply Stein's method in the second
Wiener chaos it suffices, in a sense, to provide the estimates
\eqref{eq:4} required by Assumption~\ref{assu:1}. 
\begin{rem}
  This idea is at the heart e.g.~of \cite{thale} where the authors
  could make use of the estimates provided by \cite{g-variance-gamma}
  to apply the above plan to targets as in Example~\ref{ex:eicheltha}
  (and, more generally, to variance-gamma distribution).
\end{rem}


There are two major flaws to this approach and, therefore, to the
classical take on Stein's method as adapted to the second Wiener
chaos. First the bounds required in \eqref{eq:4} can only be obtained
by solving $q$-th degree inhomogeneous equations and it is unlikely
that a unified approach will allow to deal with all targets of the
form \eqref{target-wiener1} in one sweep. Indeed different ranges of
$\alpha_i$'s imply very different properties for the corresponding
$F_{\infty}$ and hence each application of Stein's method to these
targets will necessitate ad-hoc target specific solutions, which in
the current state of knowledge we do not even know to be
bounded. Second, the bounds on higher order derivatives of $f_h$ will
necessarily depend on smoothness assumptions on the test functions
$h$; hence several important distances of the form \eqref{eq:6} with
non-smooth $h$ cannot be tackled via Stein's method.

\begin{rem}
  Such a flaw was already noted for multivariate Gaussian
  approximation, see \cite{CM}, where estimates in total variation
  distance were realized to be beyond the scope of the classical
  approaches to the method.  Recently an original solution was
  proposed in \cite{NPS} wherein a class of random vectors $F$ was
  identified to which one could apply what we will coin an
  \emph{information theoretic approach to Stein's method}; this
  resulted in a general fourth moment bound on Total Variation
  distance for multivariate normal approximation for random vectors
  $F$ satisfying a very particular integrability constraint (see
  \cite[Condition (2.53)]{NPS}).
\end{rem}



\section{Bypassing the Stein's method}\label{sec:preliminaries}

As mentioned in the conclusion to the previous section, bounding the
solutions of higher order Stein's equations is an extremely hard
task. However, if the appoximating sequence has the shape of a
weighted sum of i.i.d. random variables, we provide a new strategy to
bound the two Wasserstein distance between the sequence and the
target. Not only we completely bypass the major difficulty of bounding
the stein's solution but we can deal with distances which cannot be
reached by usual Stein's tools.

\begin{definition}
Fix $p\geq 1$. Given two probability measures $\nu$ and $\mu$ on the Borel sets of $\R^d$ whose marginals have finite absolute moments
of order $p$, define the {\it Wasserstein distance} (of order $p$) between $\nu$ and $\mu$ as the quantity
$$
{{\bf \rm W}}_p(\nu,\mu) \, = \, \inf_\pi 
\bigg ( \int_{\R^d\times\R^d} |x-y|^p d\pi(x,y)\bigg ) ^{1/p}
$$
where the infimum runs over all probability measures $\pi$ on
$\R^d\times \R^d$ with marginals $\nu$ and $\mu$.   
\end{definition}
\begin{rem}
The Wasserstein
metric may be equivalently defined by
$${{\bf \rm W}}_p(\nu,\mu) \, = \, \big( \inf \E \Vert X-Y\Vert_{d}^p
\big)^{1/p}$$
where the infimum is taken over all joint distributions of the random
variables $X$ and $Y$ with marginals $\mu$ and $\nu$ respectively, and
$\Vert \, \Vert_d$ stands for the Euclidean norm on $\R^d$. It is also
well-known that convergence with respect to ${\bf \rm W}_p$ is
equivalent to the usual weak convergence of measures plus convergence
of the first $p$th moments. Also, a direct application of H\"older
inequality implies that if $1\le p \le q$ then
${\bf \rm W}_p \le {\bf \rm W}_q$. Relevant information about
Wasserstein distances can be found, e.g. in \cite{villani-book}.  
\end{rem}

\subsection{Main result}\label{s:main-result}
Throughout this section, we assume that $\{ W_k\}_{k \ge 1}$ is a
general sequence of independent and identically distributed random
variables having {\bf finite moments} of any order such that
$\E[W_1]=0$ and $\E[W^2_1]=1$, and moreover $\kappa_r(W_1) \neq 0$ for
all $r=2,\cdots,q+1$. For a given sequence
$\{\alpha_{n,k}\}_{n,k\ge 1} \subset \R$, we assume that each element
$F_n$ of the approximating sequence is of the form

\begin{equation}\label{eq:app-general-form}
F_n= \sum_{k\ge1} \alpha_{n,k} \, W_k \qquad n\in \N.
\end{equation}
Similarly, we assume that the target random variable may be written in the following way
\begin{equation}\label{target-wiener}
  F_{\infty}:= \sum_{k=1}^q \alpha_{\infty,k} W_k.
\end{equation}
where $q \ge 2$ and all the coefficients $\{ \alpha_{\infty,k}\}_{1\le k \le q} \subset \R$ are both {non-zero} and {pairwise distinct} real numbers. Also, without loss of generality, we assume that we are dealing with normalized random variables, meaning that

\begin{equation}\label{eq:2moment-condition}
  \E[F^2_{n}]=\sum_{k\ge1} \alpha^2_{n,k}=1 \quad \text{and} \quad
  \E[F^2_\infty]
  =\sum_{k=1}^{q} \alpha^2_{\infty,k}=1, \quad \forall \, n \in \N.
\end{equation}
For any $n \in \N$, we introduce, in the shape of a lemma, a crucial quantity that can be also written as a finite linear combination of the first $2q+2$ cumulants of the random variable $F_n$ of the approximating sequence.
\begin{lem}\label{lem:linear-combiniation}
For any $n \in \N$ we have
\begin{equation} \label{eq:Delta-general}
\begin{split}
\Delta(F_n,F_\infty)=\Delta(F_n):&= \sum_{k\ge1} \alpha^2_{n,k} \prod_{r=1}^{q} \left( \alpha_{n,k} - \alpha_{\infty,r} \right)^2\\
&=\sum_{r=2}^{2q+2} \Theta_r \sum_{k\ge1} \alpha^r_{n,k}\\
&=\sum_{r=2}^{2q+2} \frac{\Theta_r}{\kappa_r(W_1)} \kappa_r(F_n).
\end{split}
\end{equation}
where the coefficients $\Theta_r$ are the coefficients of the polynomial
\begin{equation}\label{eq:polynomial}
Q(x)=(P(x))^2= (x \prod_{i=1}^q (x-\alpha_{\infty,i}))^2.
\end{equation}
\end{lem}
Now, we are ready to sate our main results.
\begin{thm}\label{thm:main-theorem1}
Let all notations and assumption in above are prevail. Then for some constant $C>0$ depending only on target random variable $F_\infty$ (and hence independent of $n$) so that 
\begin{equation}\label{eq:main-estimate-bis}
d_{{\bf \rm{W}}_2} (F_n,F_\infty) \le C \,\left(\sqrt{\Delta(F_n)}+\sum_{r=2}^{q+1}|\kappa_r(F_n)-\kappa_r(F_\infty)|\right) \qquad \forall n\ge1.
\end{equation}
More precisely, there exists a threshold $U_\infty>0$ and a constant $C>0$ depending only on the target random variable $F_\infty$ (and hence independent of $n$) such that the next bound
\begin{equation}\label{eq:upper-threshold}
\sqrt{\Delta(F_n)} +\sum_{r=2}^{q+1} \vert \kappa_r(F_n) - \kappa_r(F_\infty)\vert \le U_\infty,
\end{equation}
implies

\begin{equation}\label{eq:main-estimate}
d_{{\bf \rm{W}}_2} (F_n,F_\infty) \le C \,\sqrt{\Delta(F_n)}.
\end{equation}
In particular, if $\Delta(F_n) \to 0$ and moreover $\kappa_r(F_n) \to
\kappa_r(F_\infty)$ for all $r=2,\cdots,q+1$, then the threshold
requirement $(\ref{eq:upper-threshold})$ takes place and therefore the
sequence $F_n$ converges in distribution towards target random
variable $F_\infty$ at rate $\sqrt{\Delta(F_n)}$.  
\end{thm}

\begin{thm}\label{thm:main-theorem2}
  Let all the notations and assumptions of Theorem
  \ref{thm:main-theorem1} prevail. Assume further
  that 
  $\dim_{\mathbb{Q}} \text{span} \left\{
    \alpha^2_{\infty,1},\cdots,\alpha^2_{\infty,q}\right \}=q$.
  Then there exists a constant $C>0$ depending only on the target
  random variable $F_\infty$ such that

\begin{equation}\label{eq:main-estimate2}
d_{{\bf \rm{W}}_2} (F_n,F_\infty) \le C \,\sqrt{\Delta(F_n)}.
\end{equation}
\end{thm}

\begin{rem}\label{rem:span1}{ \rm
We remark that the condition $\dim_{\mathbb{Q}} \text{span} \left\{ \alpha^2_{\infty,1},\cdots,\alpha^2_{\infty,q}\right \}=q$ implies that if 
$\sum_{i=1}^{q+1} n_i  \alpha^2_{\infty,i}=1$ for some $n_i \in \N_0=\N \cup \{ 0\}$, then $n_1=\cdots=n_{q+1}=1$ according to normalization assumption $(\ref{eq:2moment-condition})$. We 
will use this useful observation in the proof of Theorem \ref{thm:main-theorem2}.
}
\end{rem}
\begin{rem}\label{rem:span3}{ \rm In light of
    Theorem~\ref{thm:main-theorem1}, it appears that if
    $\dim_{\mathbb{Q}} \text{span} \left\{
      \alpha^2_{\infty,1},\cdots,\alpha^2_{\infty,q}\right \}\neq q$
    then even a standard application of Stein's method (that is, using
    Stein equations and hypothetical bounds on the solutions) would
    require the control of more moments than only the first two in order to
    bound the Wasserstein-1 distance.  }
\end{rem}

\begin{rem}\label{rem:span2}{\rm
According to Theorem \ref{thm:main-theorem2} when $\dim_{\mathbb{Q}}
\text{span} \left\{
  \alpha^2_{\infty,1},\cdots,\alpha^2_{\infty,q}\right \}=q$, one can
drop the assumption of the separate convergences of the first $q+1$
cumulants, $\kappa_r(F_n) \to \kappa_r(F_\infty)$ for all
$r=2,\cdots,q+1$, in Theorem \ref{thm:main-theorem1} and hence the
only requirement $\Delta(F_n) \to 0$ implies the convergence in
distribution of the sequence $F_n$ towards the target random variable
$F_\infty$. For example, let $q=2$ and 
\begin{enumerate}
\item[(1)] $\alpha_{\infty,1} \in \mathbb{Q}$. Then, obviously $\alpha_{\infty,2} \in \mathbb{Q}$, and therefore $\dim_{\mathbb{Q}} \text{span}
\{ \alpha^2_{\infty,1},\alpha^2_{\infty,2} \}=1 \neq q=2$. Hence, for
convergence in distribution of the sequence $F_n$ towards the target
random variable $F_\infty$ in addition to convergence $\Delta(F_n) \to
0$ one needs also the convergence $\kappa_3(F_n) \to
\kappa_3(F_\infty)$, see also Remark \ref{rem:thale}.  
\item[(2)] $\alpha_{\infty,1}\in \R - \mathbb{Q}$ be a irrational number, and therefore according to normalization assumption $(\ref{eq:2moment-condition})$ the coefficient $\alpha_{\infty,2}$ will be also an irrational number. In this case, we have $\dim_{\mathbb{Q}} \text{span} \left\{ \alpha^2_{\infty,1},\alpha^2_{\infty,2}\right \}=q=2$. Hence, the sole requirement $\Delta(F_n) \to 0$ is enough for convergence in distribution of the sequence $F_n$ towards the target random variable $F_\infty$.
\end{enumerate}

}
\end{rem}

\subsection{Idea behind the  proof}
\label{sec:idea-behind-proof}

The main idea leading the proof relies on the following non trivial
observation: $F_n=\sum_k \alpha_{n,k} W_k$ converges in distribution
towards $F_\infty=\sum_{k=1}^q \alpha_{\infty,k} W_k$ if and only if
one can find $q$ coefficients among the sequence $\{\alpha_{n,k}\}_{k\ge 1}$
which are very close to the corresponding terms
$\{\alpha_{\infty,k}\}_{1\le k \le q}$ and if the $l^2$-norm of the
remaining terms is small. The main difficulty is to quantify this
phenomenon by only using the Stein discrepancy
$\Delta(F_n,F_\infty)$. To reach this goal, we proceed in several
steps which are sketched below:
\begin{itemize}
\item \text{Step 1}: Without loss of generality, we can assume that,
  for each $n$, the sequence $|\alpha_{n,k}|$ decreases with
  $k$. Indeed, this will be useful to identify which coefficients have a
  non-zero limit and which coefficients go to zero.
\item \text{Step 2}: Here we prove several inequalities which, roughly
  speaking, express the fact that if for some $k$ the coefficient
  $|\alpha_{n,k}|$ is bounded from below by a fixed constant, then it
  is necessarily close to one of the limiting coefficients
  $\alpha_{\infty,k}$.
\item \text{Step 3}: Here we prove that a certain number of the coefficients
  $\alpha_{n,k}$ (say $\alpha_{n,1},\cdots,\alpha_{n,l_n}$) are all
  close to one element (say $\alpha_\infty(n,k)$) of the set
  $\{\alpha_{\infty,1},\cdots,\alpha_{\infty,q}\}$ with possible
  repetitions. The difficult part consists in showing that
  $\sum_{k> l_n} \alpha_{n,k}^2$ is small. To do so, it is equivalent
  to prove that $\sum_{k=1}^{l_n} \alpha_{n,k}^2$ is close to
  one. Having in mind that
  $\sum_{k=1}^{l_n} |\alpha_{n,k}-\alpha_\infty(n,k)|$ is small, the
  latter claim is in turn equivalent to
  $\sum_{k=1}^{l_n}\alpha_\infty(n,k)^2=1$. We argue by contradiction
  using a maximality argument.
\item \text{Step 4}: As we said in step 3, there might be repetitions
  among the coefficients $\alpha_\infty(n,k)$ and they may also vary
  with $n$. However, if the coefficients
  $\{\alpha_{\infty,k}^2\}_{1\le k \le q}$ are rationally independent,
  then there is only one way to chose $\alpha_\infty(n,k)$ in the set
  $\{\alpha_{\infty,1},\cdots,\alpha_{\infty,q}\}$ such that
  $\sum_{k=1}^{l_n}\alpha_\infty^2(n,k)=1$. This is the idea behind
  Theorem \ref{thm:main-theorem2}. However, when we do not assume
  rational independence of the coefficients, we need to use the
  assumption of the convergence of the cumulants through a Vandermonde
  argument to proceed.
\end{itemize}

\subsection{Applications: second Wiener chaos}\label{s:applications}

In this section, we apply our main results in a desirable framework when the approximating sequence $F_n$ are elements of the second Wiener chaos of the isonormal process $\rm X=\{X(h); \ h \in \HH\}$ over a separable Hilbert space $\HH$.  We refer the reader to \cite{n-pe-1} Chapter 2 for a detailed discussion on this topic. Recall that the elements in the second Wiener chaos are random variables having the general form $F=I_2(f)$, with $f \in  \HH^{\odot 2}$. Notice that, if $f=h\otimes h$, where $h \in \HH$ is such that $\Vert h \Vert_{\HH}=1$, then using the multiplication formula one has $I_2(f)=\rm X (h)^2 -1  \stackrel{\text{law}}{=} N^2 -1$, where $N \sim \mathscr{N}(0,1)$. To any kernel $f \in \HH^{\odot 2}$, we associate the following \textit{Hilbert-Schmidt} operator
\begin{equation*}
A_f : \HH \mapsto \HH; \quad g \mapsto f\otimes_1 g. 
\end{equation*}
We also write $\{\alpha_{f,j}\}_{j \ge 1}$ and $\{e_{f,j}\}_{j \ge 1}$, respectively, to indicate the (not necessarily distinct) eigenvalues of $A_f$ and the corresponding eigenvectors. The next proposition gathers some relevant properties of the elements of the second Wiener chaos associated to $\rm X$.

\begin{prop}[See Section 2.7.4 in \cite{n-pe-1} and Lemma 3.1 in \cite{a-p-p} ] \label{second-property}
Let $F=I_{2}(f)$, $f \in \HH^{ \odot 2}$, be a generic element of the second Wiener chaos of $\rm X$, and write $\{\alpha_{f,k}\}_{k\geq 1}$ for the set of the eigenvalues of the 
associated Hilbert-Schmidt operator $A_f$.

\begin{enumerate}
 \item The following equality holds: $F=\sum_{k\ge 1} \alpha_{f,k} \big( N^2_k -1 \big)$, where $\{N_k\}_{k \ge 1}$ is a sequence of i.i.d. $\mathscr{N}(0,1)$ random variables that are elements of the isonormal process $\rm X$, and the series converges in $L^2$ and almost surely.
 \item For any $r\ge 2$,
 \begin{equation*}
  \kappa_r(F)= 2^{r-1}(r-1)! \sum_{k \ge 1} \alpha_{f,k}^r.
 \end{equation*}
\item For polynomial $Q$ as in $(\ref{eq:polynomial})$ we have $\Delta(F) = \sum_{k\geq 1} Q(\alpha_{f,k})$. In particular $\Delta(F_\infty)=0$.
\end{enumerate}
\end{prop}

The next corollary is a direct application of our main findings, namely Theorem \ref{thm:main-theorem1} and Theorem \ref{thm:main-theorem2} and provides quantitative bounds for the main results in \cite{n-po-1,a-p-p}.
\begin{cor}\label{cor:2wiener}
Assume that the normalized sequence $F_n=\sum_{k\ge 1} \alpha_{n,k} \big( N^2_k -1 \big)$ belongs to the second Wiener chaos associated to the isonormal process $\rm X$, and the target random variable $F_\infty$ as in $(\ref{target-wiener})$. Then there exists a constant $C>0$ depending only on target random variable $F_\infty$ (and hence independent of $n$) such that 
\begin{enumerate}
\item[(a)] 
$$d_{{\bf \rm W}_2}(F_n,F_\infty) \le \, C \, \bigg( \sqrt{\Delta(F_n)} + \sum_{r=2}^{q+1} \vert \kappa_r(F_n) - \kappa_r(F_\infty) \vert \bigg).$$
\item[(b)] if moreover $\dim_{\mathbb{Q}} \text{span} \{ \alpha^2_{\infty,1},\cdots,\alpha^2_{\infty,q} \} =q$, then $d_{{\bf \rm W}_2}(F_n,F_\infty) \le \, C \, \sqrt{\Delta(F_n)}$. This implies that the sole convergence $\Delta(F_n) \to \Delta(F_\infty)=0$ is sufficient for convergence in distribution towards the target random variable $F_\infty$.
\end{enumerate}
\end{cor}
\begin{rem}\label{rem:thale}{ \rm
The upper bound in Corollary \ref{cor:2wiener}, part (a) requires the separate convergences of the first $q+1$ cumulants for the convergence in distribution towards the target random variable $F_\infty$ as soon as $\dim_{\mathbb{Q}} \text{span} \{ \alpha^2_{\infty,1},\cdots,\alpha^2_{\infty,q} \} < q$. This is in very consistent with a quantitative result in \cite{thale}.  In fact, when $q=2$ and $\alpha_{\infty,1}=- \alpha_{\infty,2}=1/2$, then the target random variable $F_\infty$ $( = N_1 \times N_2$, where $N_1,N_2 \sim \mathscr{N}(0,1)$ are independent and equality holds in law) belongs to the class of {\it Variance--Gamma} distributions $VG_c(r,\theta,\sigma)$ with parameters $r=\sigma=1$ and $\theta=0$. Then, \cite[Theorem 5.10, part (a)]{thale} reads 

\begin{equation}\label{eq:thale-bound}
d_{{\bf \rm W}_1} (F_n,F_\infty) \le C\, \sqrt{\Delta(F_n) + 1/4 \, \kappa^2_3(F_n)}.
\end{equation}
Therefore, for the convergence in distribution of the sequence $F_n$ towards the target random variable $F_\infty$ in addition to convergence $\Delta(F_n) \to \Delta(F_\infty)=0$ one needs also the convergence of the third cumulant $\kappa_3(F_n) \to \kappa_3(F_\infty)=0$. Also note that in this case we have $\dim_{\mathbb{Q}} \text{span} \{ \alpha^2_{\infty,1}, \alpha^2_{\infty,2} \} =1 < q=2$. 
}
\end{rem}

\begin{ex}\label{ex:thale}{ \rm
Again assume that $q=2$ and  $\alpha_{\infty,1}=- \alpha_{\infty,2}=1/2$. Consider the fixed sequence $$F_n=F=\alpha_{\infty,1} (N^2_1 -1) - \alpha_{\infty,2} (N^2_2 -1) \qquad n \ge1.$$ 
Then $\kappa_{2r}(F_n)=\kappa_{2r}(F_\infty)$ for all $r\ge1$,  in particular $\kappa_2(F_n)=\kappa_2(F_\infty)=1$, and $\Delta(F_n)=\Delta(F_\infty)=0$. However, it is easy to see that 
the sequence $F_n$ does not converges in distribution towards the target random variable $F_\infty$, because $2=\kappa_3(F_n)  \nrightarrow  \kappa_3(F_\infty)=0$. Again, we would 
like to stress that in this case we have $\dim_{\mathbb{Q}} \text{span} \{ \alpha^2_{\infty,1}, \alpha^2_{\infty,2} \} =1 < q=2$. Therefore the requirement of separate convergences 
of the first $q+1$ cumulants is essential in Theorem \ref{thm:main-theorem1} as soon as $\dim_{\mathbb{Q}} \text{span} \{ \alpha^2_{\infty,1}, \cdots, \alpha^2_{\infty,q} \}  <  q$.
} 
\end{ex}
%
\begin{ex}{\bf [Bai--Taqqu Theorem, 2015]} { \rm
We conclude this section with a more ambitious example, providing rates of
convergence in a recent result given by \cite[Theorem 2.4]{b-t}. We
stress that many more examples and situations could be tackled by our
method. Let $Z_{\gamma_1,\gamma_2}$ be the random variable defined by:
 \begin{align*}
  Z_{\gamma_1,\gamma_2}=\int_{\mathbb{R}^2}\bigg(\int_0^1(s-x_1)^{\gamma_1}_+(s-x_2)^{\gamma_2}_+ds\bigg)dB_{x_1}dB_{x_2},
 \end{align*}
 with $\gamma_i\in(-1,-1/2)$ and $\gamma_1+\gamma_2>-3/2$. By Proposition 3.1 of \cite{b-t}, we have the following formula for the cumulants of $Z_{\gamma_1,\gamma_2}$:
 \begin{equation*}
 \kappa_m\big(Z_{\gamma_1,\gamma_2}\big) =\frac{1}{2}(m-1)!A(\gamma_1,\gamma_2)^mC_m(\gamma_1,\gamma_2,1,1)
 \end{equation*}
  where, 
\begin{align*}
A(\gamma_1,\gamma_2)&=[(\gamma_1+\gamma_2+2)(2(\gamma_1+\gamma_2)+3)]^{\frac{1}{2}}\\
&\times[B(\gamma_1+1,-\gamma_1-\gamma_2-1)B(\gamma_2+1,-\gamma_1-\gamma_2-1)\\
&+B(\gamma_1+1,-2\gamma_1-1)B(\gamma_2+1,-2\gamma_2-1)]^{-\frac{1}{2}},\\
C_m(\gamma_1,\gamma_2,1,1)&=\sum_{\sigma\in\{1,2\}^m}\int_{(0,1)^m}\prod_{j=1}^m[(s_j-s_{j-1})^{\gamma_{\sigma_j}+\gamma_{\sigma'_{j-1}}+1}_+B(\gamma_{\sigma'_{j-1}}+1,-\gamma_{\sigma_j}-\gamma_{\sigma'_{j-1}}-1)\\
&+(s_{j-1}-s_j)^{\gamma_{\sigma_j}+\gamma_{\sigma'_{j-1}}+1}_+B(\gamma_{\sigma_j}+1,-\gamma_{\sigma_{j}}-\gamma_{\sigma'_{j-1}}-1)]ds_1...ds_m,\\
&B(\alpha,\beta)=\int_0^1u^{\alpha-1}(1-u)^{\beta-1}du.
\end{align*}
Let $\rho\in(0,1)$ and $Y_{\rho}$ be the random variable defined by:
\begin{align*}
Y_{\rho}=\dfrac{a_{\rho}}{\sqrt{2}}(Z_1^2-1)+\dfrac{b_{\rho}}{\sqrt{2}}(Z_2^2-1),
\end{align*}
with $Z_i$ independent standard normal random variables and $a_{\rho}$ and $b_{\rho}$ defined by:
\begin{align*}
&a_{\rho}=\dfrac{(\rho+1)^{-1}+(2\sqrt{\rho})^{-1}}{\sqrt{2(\rho+1)^{-2}+(2\rho)^{-1}}},\\
&b_{\rho}=\dfrac{(\rho+1)^{-1}-(2\sqrt{\rho})^{-1}}{\sqrt{2(\rho+1)^{-2}+(2\rho)^{-1}}}.
\end{align*}
For simplicity, we assume that $\gamma_1\geq \gamma_2$ and $\gamma_2=(\gamma_1+1/2)/\rho-1/2$. Then \cite[Theorem 2.4]{b-t} implies that as $\gamma_1$ tends to $-1/2$:
\begin{align}\label{eq:bai-taqqu}
Z_{\gamma_1,\gamma_2}\stackrel{\text{law}}{\to} Y_{\rho}.
\end{align}
 Note that, in this case, $\gamma_2$ automatically tends to $-1/2$ as well. To prove the previous result, the authors of \cite{b-t} prove the following convergence result:
\begin{align*}
\forall m\geq 2,\ \kappa_m\big(Z_{\gamma_1,\gamma_2}\big)\rightarrow\kappa_m\big(Y_{\rho}\big)=2^{\frac{m}{2}-1}(a_{\rho}^m+b_{\rho}^m)(m-1)!.
\end{align*} 
Now, using Corollary \ref{cor:2wiener} and applying Lemma \ref{lem:cumasym}, we can present the following quantative bound for convergence $(\ref{eq:bai-taqqu})$, namely as $\gamma_1$ 
tends to $-1/2$:
\begin{equation*}
\ud_{W_2} (Z_{\gamma_1,\gamma_2},Y_{\rho}) \le C_\rho \, \sqrt{- \gamma_1 - \frac{1}{2}},
\end{equation*}
where $C_{\rho}$ is some strictly positive constant depending on $\rho$ uniquely.
}
\end{ex}
In order to apply Corollary \ref{cor:2wiener} to obtain an explicit rate for convergence $(\ref{eq:bai-taqqu})$, we need to know at which speed $\kappa_m\big(Z_{\gamma_1,\gamma_2}\big)$ 
converges towards $\kappa_m\big(Y_{\rho}\big)$. For this purpose, we have the following lemma:
\begin{lem}\label{lem:cumasym}
Under the above assumptions, for any $m\geq 3$, we have, as $\gamma_1$ tends to $-1/2$:
\begin{align*}
\kappa_m\big(Z_{\gamma_1,\gamma_2}\big)=\kappa_m\big(Y_{\rho}\big)+O\big(-\gamma_1-\frac{1}{2}\big)
\end{align*}
\end{lem}
\begin{proof}
First of all, we note that, as $\gamma_1$ tends to $-1/2$:
\begin{align*}
A(\gamma_1,\gamma_2)&=[(\gamma_1+\frac{1}{\rho}(\gamma_1+\frac{1}{2})+\frac{3}{2})(2\gamma_1+\frac{2}{\rho}(\gamma_1+\frac{1}{2})+2)]^{\frac{1}{2}}\\
&\times[B(\gamma_1+1,-(1+\frac{1}{\rho})(\gamma_1+\frac{1}{2}))B(\frac{1}{\rho}(\gamma_1+\frac{1}{2})+\frac{1}{2},-(1+\frac{1}{\rho})(\gamma_1+\frac{1}{2}))\\
&+B(\gamma_1+1,-2\gamma_1-1)B(\frac{1}{\rho}(\gamma_1+\frac{1}{2})+\frac{1}{2},-\frac{2}{\rho}(\gamma_1+\frac{1}{2}))]^{-\frac{1}{2}},\\
&\approx \dfrac{(-\gamma_1-1/2)}{\sqrt{(1+\frac{1}{\rho})^{-2}+(\frac{4}{\rho})^{-1}}}-C_{\rho}(-3+2\gamma+2\psi\big(\frac{1}{2}\big))(\gamma_1+\frac{1}{2})^2\\
&+o((-\gamma_1-1/2)^2),
\end{align*}
where $\gamma$ is the Euler constant, $\psi(.)$ is the Digamma function and $C_{\rho}$ some strictly positive constant depending on $\rho$ uniquely. Note that $-3+2\gamma+2\psi\big(1/2\big)<0$. Moreover, we have:
\begin{align}
\nonumber C_m(\gamma_1,\gamma_2,1,1)&\approx \sum_{\sigma\in\{1,2\}^m}\int_{(0,1)^m}\prod_{j=1}^m\bigg\{\mathbb{I}_{s_j>s_{j-1}}\bigg[(-\gamma_{\sigma_j}-\gamma_{\sigma'_{j-1}}-1)^{-1}-\log(s_j-s_{j-1})+(-\gamma-\psi\big(\frac{1}{2}\big))\\
\nonumber &+o(1) \bigg]+\mathbb{I}_{s_j<s_{j-1}}\bigg[ (-\gamma_{\sigma_j}-\gamma_{\sigma'_{j-1}}-1)^{-1}-\log(s_{j-1}-s_j)+(-\gamma-\psi\big(\frac{1}{2}\big))+o(1)\bigg]\bigg\}ds_1...ds_m\\
\nonumber &\approx \sum_{\sigma\in\{1,2\}^m}\int_{(0,1)^m}\prod_{j=1}^m\bigg[(-\gamma_{\sigma_j}-\gamma_{\sigma'_{j-1}}-1)^{-1}+\mathbb{I}_{s_j>s_{j-1}}\log((s_j-s_{j-1})^{-1})\\
&+\mathbb{I}_{s_j<s_{j-1}}\log((s_{j-1}-s_j)^{-1})+(-\gamma-\psi\big(\frac{1}{2}\big))+o(1)\bigg]ds_1...ds_m.\label{approxCm}
\end{align}
Note that $-\gamma-\psi\big(\frac{1}{2}\big)>0$. The diverging terms in $C_m(\gamma_1,\gamma_2,1,1)$ are $B(\gamma_{\sigma'_{j-1}}+1,-\gamma_{\sigma_j}-\gamma_{\sigma'_{j-1}}-1)$ and $B(\gamma_{\sigma_j}+1,-\gamma_{\sigma_{j}}-\gamma_{\sigma'_{j-1}}-1)$. At $\sigma$ and $j$ fixed, the only possible values are:
\begin{align*}
&B(\gamma_1+1,-\gamma_{1}-\gamma_{2}-1)=B(\gamma_1+1,-(\gamma_1+\frac{1}{2})(1+\frac{1}{\rho})),\\
&\approx -\dfrac{1}{(1+\frac{1}{\rho})(\gamma_1+\frac{1}{2})}+(-\gamma-\psi(\frac{1}{2}))+o(1),\\
&B(\gamma_2+1,-\gamma_{1}-\gamma_{2}-1)=B(\frac{1}{\rho}(\gamma_1+\frac{1}{2})+\frac{1}{2},-(\gamma_1+\frac{1}{2})(1+\frac{1}{\rho})),\\
&\approx -\dfrac{1}{(1+\frac{1}{\rho})(\gamma_1+\frac{1}{2})}+(-\gamma-\psi(\frac{1}{2}))+o(1),\\
&B(\gamma_1+1,-2\gamma_{1}-1)\approx -\dfrac{1}{2(\gamma_1+\frac{1}{2})}+(-\gamma-\psi(\frac{1}{2}))+o(1),\\
&B(\gamma_2+1,-2\gamma_{2}-1)=B(\frac{1}{\rho}(\gamma_1+\frac{1}{2})+\frac{1}{2},-\frac{2}{\rho}(\gamma_1+\frac{1}{2})),\\
&\approx -\dfrac{\rho}{2(\gamma_1+\frac{1}{2})}+(-\gamma-\psi(\frac{1}{2}))+o(1).
\end{align*}
Moreover, we have, for $j$ fixed:
\begin{align*}
(s_j-s_{j-1})^{\gamma_{\sigma_j}+\gamma_{\sigma'_{j-1}}+1}_+&=\mathbb{I}_{s_j>s_{j-1}}(s_j-s_{j-1})^{\gamma_{\sigma_j}+\gamma_{\sigma'_{j-1}}+1}\\
&\approx \mathbb{I}_{s_j>s_{j-1}}[1+\log(s_j-s_{j-1})(\gamma_{\sigma_j}+\gamma_{\sigma'_{j-1}}+1)
+o((\gamma_{\sigma_j}+\gamma_{\sigma'_{j-1}}+1))].
\end{align*}
Developing the product in the right hand side of (\ref{approxCm}), we obtain: 
\begin{align*}
C_m(\gamma_1,\gamma_2,1,1)&\approx \sum_{\sigma\in\{1,2\}^m}\prod_{j=1}^m(-\gamma_{\sigma_j}-\gamma_{\sigma'_{j-1}}-1)^{-1}\\
&+(-\gamma-\psi\big(\frac{1}{2}\big))\sum_{\sigma\in\{1,2\}^m}\sum_{j=1}^m\prod_{k=1,\ k\ne j}^m(-\gamma_{\sigma_k}-\gamma_{\sigma'_{k-1}}-1)^{-1}\\
&+\sum_{\sigma\in\{1,2\}^m}\sum_{j=1}^m\prod_{k=1,\ k\ne j}^m(-\gamma_{\sigma_k}-\gamma_{\sigma'_{k-1}}-1)^{-1}\int_{(0,1)^m}\bigg[\mathbb{I}_{s_j>s_{j-1}}\log((s_j-s_{j-1})^{-1})\\
&+\mathbb{I}_{s_j<s_{j-1}}\log((s_{j-1}-s_j)^{-1})\bigg]ds_1...ds_m+o((-\gamma_1-\frac{1}{2})^{-m+1})\\
&\approx \sum_{\sigma\in\{1,2\}^m}\prod_{j=1}^m(-\gamma_{\sigma_j}-\gamma_{\sigma'_{j-1}}-1)^{-1}\\
&+(-\gamma-\psi\big(\frac{1}{2}\big)+\frac{3}{2})\sum_{\sigma\in\{1,2\}^m}\sum_{j=1}^m\prod_{k=1,\ k\ne j}^m(-\gamma_{\sigma_k}-\gamma_{\sigma'_{k-1}}-1)^{-1}\\
&+o((-\gamma_1-\frac{1}{2})^{-m+1})\\
\end{align*}
This leads to the following asymptotic for the cumulants of $Z_{\gamma_1,\gamma_2}$,
\begin{align*}
\kappa_m\big(Z_{\gamma_1,\gamma_2}\big)&\approx \frac{(m-1)!}{2}\bigg[\dfrac{(-\gamma_1-1/2)}{\sqrt{(1+\frac{1}{\rho})^{-2}+(\frac{4}{\rho})^{-1}}}-C_{\rho}(-3+2\gamma+2\psi\big(\frac{1}{2}\big))(\gamma_1+\frac{1}{2})^2\\
&+o((-\gamma_1-1/2)^2)\bigg]^m\bigg[\sum_{\sigma\in\{1,2\}^m}\prod_{j=1}^m(-\gamma_{\sigma_j}-\gamma_{\sigma'_{j-1}}-1)^{-1}\\
&+(-\gamma-\psi\big(\frac{1}{2}\big)+\frac{3}{2})\sum_{\sigma\in\{1,2\}^m}\sum_{j=1}^m\prod_{k=1,\ k\ne j}^m(-\gamma_{\sigma_k}-\gamma_{\sigma'_{k-1}}-1)^{-1}\\
&+o((-\gamma_1-\frac{1}{2})^{-m+1})\bigg],\\
&\approx \frac{(m-1)!}{2}\dfrac{(-\gamma_1-1/2)^m}{\bigg(\sqrt{(1+\frac{1}{\rho})^{-2}+(\frac{4}{\rho})^{-1}}\bigg)^m}\sum_{\sigma\in\{1,2\}^m}\prod_{j=1}^m(-\gamma_{\sigma_j}-\gamma_{\sigma'_{j-1}}-1)^{-1}\\
&+O((-\gamma_1-\frac{1}{2}))\\
&\approx \kappa_m\big(Y_{\rho}\big)+O((-\gamma_1-\frac{1}{2})),
\end{align*}
where we have used similar computations as in the proof of Theorem $2.4$ of \cite{b-t} for the last equality.
\end{proof}

\subsection{Proof of Theorem \ref{thm:main-theorem1}}

In what follows, we will need the following useful lemma. 
\begin{lem}\label{lem:appendix}
For the vector $\bm{ \alpha_\infty} = (\alpha_{\infty,1}, \cdots,\alpha_{\infty,q}) \in \R^q$ where $\alpha_{\infty,i}$ are non-zero and distinct, we denote  
\begin{equation*}
d(x,\bm{ \alpha_\infty} ):= \min_{i=1,\cdots,q} \, \vert x - \alpha_{\infty,i} \vert, \qquad \forall \, x \in \R.
\end{equation*}
Then, there exists a constant $M$ such that
\begin{equation*}
d(x, \bm{ \alpha_\infty})^2 \le M \prod_{i=1}^q (x - \alpha_{\infty,i})^2.
\end{equation*}
\end{lem}
\begin{proof}
Consider the function $f:\R - \{ \alpha_{\infty,1},\cdots, \alpha_{\infty,q} \} \to \R$ given by
\begin{equation*}
f(x):= \frac{\prod_{i=1}^q (x - \alpha_{\infty,i})^2}{d(x,  \bm{ \alpha_\infty})^2}.
\end{equation*}
Then, obviously $f$ is a continuous function on
$\R - \{ \alpha_{\infty,1},\cdots, \alpha_{\infty,q} \}$ and can be
extended to a continuous function on whole real line $\R$ by setting
$f(\alpha_{\infty,i}):= \prod_{j \neq i} (\alpha_{\infty,j} -
\alpha_{\infty,i})^2 \neq 0$
at each point $\alpha_{\infty,i}$ for $i=1,\cdots,q$. On the other
hand, note that we have $f(x) \to \infty$ as $\vert x \vert$ tends to
infinity. Hence, $f$ is bounded from below by a positive constant, say
$M$.
\end{proof}

  We split the proofs of Theorems \ref{thm:main-theorem1} and
  \ref{thm:main-theorem2} in several steps. Throughout, $C$ stands
  for a generic constant that  is independent of $n$ but  may differ
  from line to line.

\

\noindent {\bf Step 1}: (Re-ordering the coefficients) Under the
second moment conditions $(\ref{eq:2moment-condition})$, we know that
for any fixed $n \ge 1$, we have $\lim_{k \to \infty} \alpha_{n,k}=0$.
Therefore, $\max \{ \vert \alpha_{n,k} \vert \, : \, k \ge 1 \}$ is
attained in, at least one value, say $\vert \alpha_{n,k_1}\vert$.
Similarly, $\max \{ \vert \alpha_{n,k} \vert \, : \, k \neq k_1 \}$ is
attained in some value $\vert \alpha_{n,k_2}\vert$. We repeat this
procedure by induction and we may build a decreasing sequence
$\vert \alpha_{n,k_1}\vert \ge \vert \alpha_{n,k_2}\vert \ge \cdots
\vert \alpha_{n,k_p}\vert \ge \cdots$
such that for all $i \ge 1$, we have
\begin{equation*}
\vert \alpha_{n,k_i}\vert = \max \{ \vert \alpha_{n,k}\vert \, : \, k\neq k_1,k_2,\cdots,k_{i-1}\}.
\end{equation*}
Also, for all $n,p \ge 1$ we have $1 \ge \sum_{i=1}^{p}\alpha^2_{n,k_i} \ge \sum_{i=1}^{p} \alpha^2_{n,k} \to 1$ as $p$ tends to infinity. Therefore
\begin{equation}\label{eq:re-ordering}
F_n \stackrel{\text{law}}{=} \sum_{i=1}^{\infty} \alpha_{n,k_i} W_i, \qquad \forall \, n \ge 1.
\end{equation}
To emphasize the maximality property of $\alpha_{n,k_i}$, we denote
$\alpha_{n,k_i}$ by $\alpha_{\text{max}}(n,i)$, and in the rest of the
proof we assume that for each $n \ge 1$, $F_n$ is
given by the right hand side of $(\ref{eq:re-ordering})$.

\

\noindent {\bf Step 2}: (Bounding the $\alpha_{\text{max}}(n,i)$'s
from below)
%
For any $p\ge 1$, we introduce the quantity 

\begin{equation}\label{eq:Delta-p}
\Delta_p(F_n):= \sum_{k=p}^{\infty} \alpha^2_{\text{max}}(n,k) \prod_{r=1}^{q} \left( \alpha_{\text{max}}(n,k) -\alpha_{\infty,r}\right)^2.
\end{equation}
Next, we observe that
\begin{align}
\Delta_p(F_n)& = \sum_{k=p}^{\infty} Q(\alpha_{\text{max}}(n,k)) =\sum_{r=2}^{2q+2} \Theta_r \sum_{k=p}^{\infty}
  \alpha^r_{\text{max}}(n,k)
               \nonumber \\
\label{eq:product-estimate}
& = \Theta_2 \Big(1 - \sum_{k=1}^{p-1} \alpha^2_{\text{max}}(n,k)
  \Big) + \sum_{r=3}^{2q+2} \Theta_r \sum_{k=p}^{\infty}
  \alpha^r_{\text{max}}(n,k).
\end{align}

Besides, for any $r\ge3$ and the maximality property of the
coefficients $\alpha_{\text{max}}(n,k)$ together with the
normalization assumption $(\ref{eq:2moment-condition})$ we have the
following estimate (which is valid for all $r\ge 3$):
\begin{equation*}
\left\vert \sum_{k=p}^{\infty} \alpha^r_{\text{max}}(n,k) \right\vert \le       \left \vert  \alpha^{r-2}_{\text{max}}(n,p) \right \vert \times  \left( 1- \sum_{k=1}^{p-1}  \alpha^2_{\text{max}}(n,k) \right).
\end{equation*}
Also note  that, since  $ \vert \alpha_{\text{max}}(n,p) \vert \le
1$, we always have $ \vert \alpha_{\text{max}}(n,p) \vert^{r-2} \le
\vert \alpha_{\text{max}}(n,p) \vert$ (still for all $r\ge 3$). We may deduce
\begin{equation*}
\Big \vert \Delta_p(F_n)  -  \Theta_2 \Big(1 - \sum_{k=1}^{p-1} \alpha^2_{\text{max}}(n,k) \Big)
\Big \vert   \le \vert \alpha_{\text{max}}(n,p) \vert \Big(1 -
\sum_{k=1}^{p-1} \alpha^2_{\text{max}}(n,k) \Big) \sum_{r=3}^{2q+2}
\vert \Theta_r \vert,
\end{equation*}
leading in turn to the lower bound
\begin{equation}\label{eq:general-lower-estimate}
\Big \vert \alpha_{\text{max}}(n,p) \Big \vert \ge \frac{\vert \Theta_2 \vert}{\sum_{r=3}^{2q+2} \vert \Theta_r \vert} - \frac{\Delta_p(F_n)}{ \Big(1 - \sum_{k=1}^{p-1} \alpha^2_{\text{max}}(n,k) \Big) 
\times  \sum_{r=3}^{2q+2} \vert \Theta_r \vert}.
\end{equation}
Note that in the right hand side of \eqref{eq:general-lower-estimate},
the first summand depends only on the limiting law. In order to deal
with the second summand, we need control on
$1 - \sum_{k=1}^{p-1} \alpha^2_{\text{max}}(n,k)$. To this end we
introduce the following useful quantities:
\begin{equation}\label{eq:vartheta}
\vartheta = \min \Big\{ 1 - \sum_{i=1}^q n_i \alpha^2_{\infty,i} \, \Big\vert \, (n_1,\cdots,n_q) \in \N_0^q, \quad \text{and} \quad 1 - \sum_{i=1}^q n_i \alpha^2_{\infty,i} > 0 \Big\}.
\end{equation}
\begin{equation}\label{eq:varkappa}
\varkappa=\max \Big\{ 1 - \sum_{i=1}^q n_i \alpha^2_{\infty,i} \, \Big\vert \, (n_1,\cdots,n_q) \in \N_0^q, \quad \text{and} \quad 1 - \sum_{i=1}^q n_i \alpha^2_{\infty,i} < 0 \Big\}.
\end{equation}
Note that, for any vector $(n_1,\cdots,n_q) \in \N_0^q$ such that $ 1
- \sum_{i=1}^q n_i \alpha^2_{\infty,i} > 0$  we have
\begin{equation*}
\vartheta < 1 - \sum_{i=1}^q n_i \alpha^2_{\infty,i} \le 1 - \alpha^2_{\text{min}}(\infty) \left( \sum_{i=1}^q n_i \right)
\end{equation*}
with $\alpha^2_{\text{min}}(\infty)= \min \{ \alpha^2_{\infty,i} \,
: \,   i=1,\cdots,q\}$. 
This leads to the important upper bound estimate 
\begin{equation}\label{eq:repetition-upper-estimate}
\sum_{i=1}^q n_i \le \frac{1- \vartheta}{\alpha^2_{\text{min}}(\infty)},
\end{equation}
which is finite because our assumption on the coefficients of the
target random variable $F_\infty$ implies that 
$\alpha^2_{\text{min}}(\infty) \neq 0$. Finally we
set
\begin{equation*}\label{eq:L-uuper-estimate}
L:= \lfloor \frac{1-\vartheta}{\alpha^2_{\text{min}}(\infty)} \rfloor
\mbox{ and }\varrho:= \min \{ \vartheta,\vert \varkappa \vert \}.
\end{equation*}

\

\noindent {\bf Step 3}: (Induction procedure) We now aim to
use Step $2$ to control the distance between the eigenvalues
$\alpha_{\text{max}}(n,p)$ and the eigenvalues $\alpha_{\infty,p}$ of the
target random variable $F_\infty$ in terms of $\Delta(F_n)$.  Taking
into account the assumption $\Delta(F_n) \to 0$ as $n$ tends to
infinity, one can find an $N_0 \in \N$ such that
\begin{equation}\label{eq:Delta_n-estimate}
\Delta(F_n) \le \min \{ \frac{\vert \Theta_2 \vert}{2}, \frac{\varrho
  \vert \Theta_2 \vert}{4} \} \, \text{ and } \, \sqrt{L+1} \times
\sqrt{M \times \left(  \frac{2 \sum_{r=3}^{2q+2} \vert \Theta_r \vert}{\vert \Theta_2 \vert} \right)^2 \times \Delta(F_n)} < \frac{\varrho}{2}.
\end{equation}
Next, for any $n \ge N_0$ we define
\begin{equation*}
\mathscr{A}_n:= \Big \{ 1 \le p \le L  :  \text{ for any } 1 \le l \le p \text{ we have }     \vert \alpha_{\text{max}}(n,l) \vert             \ge  \frac{\vert \Theta_2 \vert}{2 \sum_{r=3}^{2q+2} \vert \Theta_r \vert}    \Big\}.
\end{equation*}
The collection $\mathscr{A}_n$ is not empty. In fact, using the
estimate $(\ref{eq:general-lower-estimate})$ with $p=1$, one can
immediately get
\begin{equation}\label{eq:alpha1-estimate}
\Big \vert \alpha_{\text{max}}(n,1) \Big\vert \ge \frac{\vert \Theta_2 \vert}{\sum_{r=3}^{2q+2} \vert \Theta_r \vert} - \frac{\Delta_1(F_n)}{  \sum_{r=3}^{2q+2} \vert \Theta_r \vert} \ge  \frac{\vert \Theta_2 \vert}{2 \sum_{r=3}^{2q+2} \vert \Theta_r \vert}
\end{equation}
because, for $n\ge N_0$, we know that
$\Delta_1(F_n) = \Delta(F_n)\le \frac{\vert \Theta_2 \vert}{2}$ thanks
to the first inequality in $(\ref{eq:Delta_n-estimate})$.  Note that
for any $p \ge 1$, we have $\Delta_p(F_n) \le \Delta(F_n) \to 0$ as
$n$ tends to infinity.  On the other hand, $\mathscr{A}_n$ is set
bounded by $L$, and therefore has a maximal element which we denote by
$L_n$. By the definitions of $\Delta(F_n)$ and of the set
$\mathscr{A}_n$ we infer that
\begin{equation*}
\sum_{k=1}^{L_n} \prod_{r=1}^{q} \left( \alpha_{\text{max}}(n,k) -
  \alpha_{\infty,r} \right)^2 \le  \left( \frac{ 2 \sum_{r=3}^{2q+2} \vert \Theta_r \vert}{\vert \Theta_2 \vert} \right)^2 \Delta(F_n).
\end{equation*}
Then, in virtue of Lemma \ref{lem:appendix} this reads 
\begin{equation}\label{eq:d-estimate}
\sum_{k=1}^{L_n} d(\alpha_{\text{max}} (n,k), \bm{ \alpha_\infty} )^2
\le M \times  \left( \frac{ 2 \sum_{r=3}^{2q+2} \vert \Theta_r \vert}{\vert
  \Theta_2 \vert}  \right)^2 \Delta(F_n).
\end{equation} 
On the other hand, for any $1 \le k \le L_n$ there exists some
$\alpha_{\infty}(n,k) \in \{ \alpha_{\infty,1}, \cdots,
\alpha_{\infty,q} \}$ realizing the minimum in definition of
$d(\alpha_{\text{max}} (n,k), \bm{ \alpha_\infty})$. Here, one has to
note that the coefficients $\alpha_\infty(n,k)$ in general can be
repeated. So, we can rewrite $(\ref{eq:d-estimate})$ as 
\begin{equation}\label{eq:after-d-estimate}
\sum_{k=1}^{L_n} \left \vert \alpha_{\text{max}} (n,k) - \alpha_{\infty}(n,k) \right \vert^2 \le M \times  \left( \frac{ 2 \sum_{r=3}^{2q+2} \vert \Theta_r \vert}{\vert
  \Theta_2 \vert}  \right)^2 \Delta(F_n),
\end{equation}
which gives us part of the control we seek. It still remains to show
that the remainder is well-behaved. To this end we will show that $\forall n \ge
N_0$ there exists $l_n \in \left\{ 1, \ldots, L+1 \right\}$ such that 
\begin{equation}\label{eq:after-d-estimate2b}
\sum_{k=1}^{\ell_n} \left \vert \alpha_{\text{max}} (n,k) -
  \alpha_{\infty}(n,k) \right \vert^2 \le M \times  \left( \frac{ 2
    \sum_{r=3}^{2q+2} \vert \Theta_r \vert}{\vert 
  \Theta_2 \vert}  \right)^2 \Delta(F_n) 
\end{equation}
and 
\begin{equation}\label{eq:after-d-estimate2a}
  \sum_{k=\ell_n+1}^{\infty} \alpha^2_{\text{max}} (n,k)  \le C \sqrt{\Delta(F_n)}.
\end{equation}
First,  taking into account the
estimate $(\ref{eq:after-d-estimate})$, one  can infer that
\begin{align}
   \left\vert (1- \sum_{k=1}^{L_n} \alpha^2_{\text{max}} (n,k) ) \right. &\left.- (1 -
   \sum_{k=1}^{L_n} \alpha^2_{\infty}(n,k)) \right\vert \nonumber\\
 & = \left\vert \sum_{k=1}^{L_n} \left(  \alpha^2_{\infty}(n,k) - \alpha^2_{\text{max}} (n,k) \right) \right\vert\nonumber \\
& \le 2 \sum_{k=1}^{L_n} \vert \alpha_{\text{max}} (n,k) - \alpha_{\infty}(n,k) \vert\nonumber \\
&\le \sqrt{L_n} \times \sqrt{ M \times \left(  \frac{ 2
  \sum_{r=3}^{2q+2} \vert \Theta_r \vert}{\vert \Theta_2 \vert}
  \right)^2 \times \Delta(F_n)}\nonumber\\ 
& \le  \sqrt{L} \times \sqrt{ M \times \left(  \frac{ 2
  \sum_{r=3}^{2q+2} \vert \Theta_r \vert}{\vert \Theta_2 \vert}
  \right)^2 \times \Delta(F_n)}.  \label{eq:variance-estimate}
\end{align}
In order to conclude we now seek for an index $l_n$ such that
\begin{equation} \label{eq:sum=0}
   \sum_{k=1}^{\ell_n} \alpha^2_{\infty} (n,k) = 1.
\end{equation}
 Given $n \ge N_0$ we have  three possibilities.   
\begin{itemize}
\item If $1- \sum_{k=1}^{L_n} \alpha^2_\infty(n,k) =0$ then we can
  take $\ell_n=L_n$ and, by \eqref{eq:variance-estimate} and
  \eqref{eq:after-d-estimate}, we are done.
\item If $1- \sum_{k=1}^{L_n} \alpha^2_\infty(n,k) >0$ then  $1- \sum_{k=1}^{L+1}
\alpha^2_{\infty}(n,k) =0$ and we can
  take $\ell_n=L+1$. Indeed we necessarily have
  $1- \sum_{k=1}^{L_n} \alpha^2_\infty(n,k) > \varrho$ by definition
  of $\varrho$. Using the
  second inequality in $(\ref{eq:Delta_n-estimate})$ as well
  as 
 the estimate given in
  $(\ref{eq:variance-estimate})$ one can infer that
\begin{equation}\label{eq:1}
1 - \sum_{k=1}^{L_n} \alpha^2_{\text{max}} (n,k) \ge \frac{\varrho}{2}.
\end{equation}
Now, using estimate $(\ref{eq:general-lower-estimate})$ with $p=L_n +1$ together with the first estimate in $(\ref{eq:Delta_n-estimate})$ we obtain that 
\begin{equation}\label{eq:last-alpha-estimate}
\vert \alpha_{\text{max}} (n,L_n +1) \vert \ge \frac{\vert \Theta_2 \vert}{2 \sum_{r=3}^{2q+2} \vert \Theta_r \vert}.
\end{equation}
If $L_n$ were strictly less than $L$, then it would contradict the
fact that $L_n$ is the maximal element of the set $\mathscr{A}_n$, and
therefore $L_n=L$. Now, following exactly the same lines as in the
beginning of this step and using $(\ref{eq:last-alpha-estimate})$ one
can infer that
\begin{enumerate}
\item[(i)]
$\sum_{k=1}^{L+1} \left \vert \alpha_{\text{max}} (n,k) -
  \alpha_{\infty}(n,k) \right\vert^2 \le M \times \left( \frac{ 2
    \sum_{r=3}^{2q+2} \vert \Theta_r \vert}{\vert \Theta_2 \vert}
\right)^2\times \Delta(F_n).$ 
\item[(ii)]$ \left\vert (1- \sum_{k=1}^{L+1} \alpha^2_{\text{max}} (n,k) ) - (1 - \sum_{k=1}^{L+1} \alpha^2_{\infty}(n,k)) \right\vert < \frac{\varrho}{2}.$
\end{enumerate}
Now, we are left to show that $1- \sum_{k=1}^{L+1}
\alpha^2_{\infty}(n,k) =0$. First, note that according to definition
of $L$, if $1- \sum_{k=1}^{L+1} \alpha^2_{\infty}(n,k) \neq 0$, then
we have to have that $1- \sum_{k=1}^{L+1} \alpha^2_{\infty}(n,k)
<0$. Now, again using definition of $\varkappa$ and $\varrho$, this
implies that $1- \sum_{k=1}^{L+1} \alpha^2_{\infty}(n,k) \le -
\varrho$. Now, taking into account $(ii)$ and the fact that $1-
\sum_{k=1}^{L+1} \alpha^2_{\text{max}} (n,k) \ge 0$, we arrive to $$
-\frac{\varrho}{2}\le1 - \sum_{k=1}^{L+1} \alpha^2_{\infty}(n,k) \le -
\varrho.$$ That is obviously a contradiction and therefore $1-
\sum_{k=1}^{L+1} \alpha^2_{\infty}(n,k) =0$. 
Finally, employing the same estimates as in
$(\ref{eq:variance-estimate})$  we get
\begin{align*}\label{intermediary-tail} 
\sum_{k \ge \ell_n +1} \alpha^2_{\text{max}}(n,k) &= \Big \vert 1 - \sum_{k \le \ell_n} \alpha^2_{\text{max}}(n,k) \Big \vert \\
&= \Big \vert \big(1 - \sum_{k \le \ell_n} \alpha^2_{\text{max}}(n,k)
\big) - \big( 1 - \sum_{k \le \ell_n} \alpha^2_{\infty}(n,k) \big)
\Big\vert \\ 
& \le C \Big( \sum_{k \le \ell_n} \vert \alpha_{\text{max}}(n,k)  -  \alpha_{\infty}(n,k) \vert^2 \Big)^{\frac{1}{2}}\\
& \le C \, \sqrt{    \Delta(F_n)}.
\end{align*}
\item The case $1-
\sum_{k=1}^{L_n} \alpha^2_\infty(n,k) <0$ can be also discussed in the
same way, this time leading to some $\ell_n < L$. 
\end{itemize}

The square root in \eqref{eq:after-d-estimate2a} is clearly not sharp
and we need to remove it.  To do this, observe that an obvious
consequence of \eqref{eq:after-d-estimate2a} is
$$\alpha^2_{\text{max}}(n,k)\le C \sqrt{\Delta(F_n)} \qquad \forall \,
k \, \ge \ell_n+1 .$$
Thus, if $\Delta(F_n)$ is small enough in the sense that
 $C \sqrt{\Delta(F_n)} \le U_\infty < \min_{1 \le i \le q+1}
 \alpha_{\infty,i}$, since  $\alpha_{\infty,i} \neq
0$  by assumption, it follows trivially that 
%
we may find a universal positive constant $C$
such that for any $n \in \N$ and for any $k\ge \ell_n+1$ it holds that
$$\prod_{i=1}^q(\alpha_{\text{max}}(n,k)-\alpha_{\infty,i})^2>C.$$
Hence
\begin{align*}
  \Delta(F_n) \ge \Delta_{\ell_n+1}(F_n) \ge C \sum_{k=\ell_n+1}^\infty \alpha^2_{\text{max}}(n,k),
\end{align*}
leading to the final estimate
\begin{equation}\label{eq:variance-tail-estimate} 
 \sum_{k=\ell_n+1}^\infty \alpha^2_{\text{max}}(n,k) \le \frac{1}{C} \, \Delta(F_n).
\end{equation}

\

\noindent \textbf{Step 4:} (An algebraic argument) We have showed that
$\forall n \ge N_0$ there exists
$l_n \in \left\{ 1, \ldots, L+1 \right\}$ such that
\begin{equation}\label{eq:after-d-estimate2b}
\sum_{k=1}^{\ell_n}  \alpha^2_{\infty}(n, k)=1 \mbox{ and }\sum_{k=1}^{\ell_n} \left \vert \alpha_{\text{max}} (n,k) -
  \alpha_{\infty}(n,k) \right \vert^2 +\sum_{k=\ell_n+1}^{\infty}
\alpha^2_{\text{max}} (n,k) \le C \Delta(F_n). 
\end{equation}
For every $n$, let
$\nu(n,k), k=1, \ldots, q$ stand for the multiplicity of the coefficient
$\alpha_{\infty,k}$ realizing the minimum in the definition of the
$d$-distance (the numbers $\nu(n,k)$ can, a priori,  be equal or take value
zero). 
Clearly $\sum_{k=1}^{q} \nu(n,k) \alpha^2_{\infty,k}=1$.  In
order to reap the conclusion in 2-Wasserstein distance, we are only
left to show that, in fact, we have that $\nu(n,k)=1$ for all
$1\le k \le q$ and $\ell_n =q$.  

\begin{proof}[Proof of Theorem~\ref{thm:main-theorem2}]
  In this case
  $\dim_{\mathbb{Q}} \text{span} \left\{
    \alpha^2_{\infty,1},\cdots,\alpha^2_{\infty,q}\right \}=q$.
  Then, condition $\sum_{k=1}^{q} \nu(n,k) \alpha^2_{\infty,k}=1$
  necessarily implies $\nu(n,k)=1$ for all $k=1, \ldots, q$ and thus
  $\ell_n=q$ (recall Remark \ref{rem:span1}).
\end{proof}

\begin{proof}[Proof of Theorem \ref{thm:main-theorem1}]
 For any $r=2,\cdots,q+1$, one can
write
\begin{equation*}
\begin{split}
\kappa_r(F_n) 
& = \kappa_r(W_1) \sum_{k=1}^{\ell_n} \alpha^r_{\text{max}}(n,k) + \kappa_r(W_1) \sum_{k=\ell_n+1}^{\infty} \alpha^r_{\text{max}}(n,k).
\end{split}
\end{equation*}
Therefore, according to $(\ref{eq:variance-tail-estimate})$, we obtain
\begin{equation}\label{eq:step4-1}
\Big\vert \kappa_r(F_n) - \kappa_r(W_1) \sum_{k=1}^{\ell_n} \alpha^r_{\text{max}}(n,k) \Big\vert  \le C  \, \Delta(F_n).
\end{equation}\label{eq:step4-2}
Moreover, from the Cauchy-Schwarz inequality and \eqref{eq:after-d-estimate2b}, we have
\begin{equation}\label{eq:step4-3}
\Big \vert  \sum_{k=1}^{\ell_n} \alpha^r_{\text{max}}(n,k) - \sum_{k=1}^{q}\nu(n,k) \alpha^r_{\infty,k} \Big \vert \le C  \,  \sqrt{\Delta(F_n)}
\end{equation}
for any $r=2,\cdots,q+1$. 
Combining estimates $(\ref{eq:step4-1})$ and $(\ref{eq:step4-3})$ together, we arrive to
\begin{equation}\label{eq:step4-4}
\Big\vert \kappa_r(F_n) - \kappa_r(W_1)  \sum_{k=1}^{q}\nu(n,k) \alpha^r_{\infty,k} \Big \vert \le  C  \, \sqrt{\Delta(F_n)}.
\end{equation}
Now, we introduce the following $q \times q$ so-called Vandermonde
matrix which is invertible because we assumed that the coefficients
$\alpha_{\infty,k}$ are pairwise distinct
\[
\mathbb{M}=
  \begin{bmatrix}
    \alpha^2_{\infty,1} & \alpha^2_{\infty,2} & \cdots & \alpha^2_{\infty,q} \\
   \cdots & \cdots & \\
  \alpha^{q+1}_{\infty,1} & \alpha^{q+1}_{\infty,2} &\cdots  &\alpha^{q+1}_{\infty,q}
  \end{bmatrix}.
\]
Set ${\bf V}_n:= (\nu(n,1),\cdots,\nu(n,q))^t$ ($t$ stands for
transposition) and consider the vector 
\begin{equation*}
  \Xi=\left( \frac{\kappa_2(F_\infty)}{\kappa_2(W_1)}, \cdots, \frac{\kappa_{q+1}(F_\infty)}{\kappa_{q+1}(W_1)}\right)^t.
\end{equation*}
 Note that by our assumption $\kappa_r(W_1) \neq 0$ for all 
$2 \le r \le q+1$. Now, inequality $(\ref{eq:step4-4})$ entails that
\begin{equation}\label{determineV-1}
\Vert \mathbb{M} {\bf V}_n -\Xi \Vert_{\infty} \le C\left( \sqrt{\Delta(F_n)}+\sum_{r=2}^{q+1}\vert \kappa_r(F_\infty)-\kappa_r(F_n)\vert\right),
\end{equation}
 which, by inversion, implies that
\begin{equation}\label{eq:determineV-2}
\Vert {\bf V}_n - \mathbb{M}^{-1}\Xi \Vert_{\infty} \le C\left( \sqrt{\Delta(F_n)}+\sum_{r=2}^{q+1}\vert \kappa_r(F_\infty)-\kappa_r(F_n)\vert\right).
\end{equation}

If the right hand side of the estimate $(\ref{eq:determineV-2})$ is strictly less than $1$, say $\frac{1}{3}$, then because ${\bf V}_n$ is a vector of integer numbers, we necessarily have
$${\bf V}_n = \lfloor \mathbb{M}^{-1}\Xi \rfloor$$
where $\lfloor \cdot \rfloor$ denotes the standard integer part. 
In order to conclude the proof, it remains to show that $ \lfloor \mathbb{M}^{-1}\Xi \rfloor$ is the vector 
$(1,\cdots,1)^t$ and $\ell_n=q$. Note that $ \lfloor \mathbb{M}^{-1}\Xi \rfloor$ does not depend on the approximating sequence $F_n$ but only on the target random variable $F_\infty$. 
So we might place ourselves in the obvious situation where the approximating sequence $F_n$ is nothing else than the target itself. In this case ${\bf V}_n=(1,\cdots,1)^t$ and necessarily

$$ \lfloor \mathbb{M}^{-1}\Xi \rfloor=(1,\cdots,1)^t.$$
\end{proof}

\begin{rem}\label{rem:true-rate}{ \rm Here, we would like to strongly
    highlight that in order to show that all the multiplicity numbers
    $\nu(n,k)=1$ for all $k=1,\cdots,q$, one needs that the all
    distances $\vert \kappa_r(F_n) - \kappa_r(F_\infty)\vert$ are very
    small (in the above sense) for all $r=2,\cdots,q+1$. However,
    after doing that taking into account the estimate
    $\eqref{eq:step4-4}$, one can immediately observe that the true
    convergence rate is given by $\sqrt{\Delta(F_n)}$.  }
\end{rem}

 \section*{Acknowledgments}
BA's research is supported by a Welcome Grant from the
 Universit\'e de Li\`ege.  YS gratefully acknowledges support from the
 IAP Research Network P7/06 of the Belgian State (Belgian Science Policy).

\addcontentsline{toc}{section}{References}%

\

(B. Arras and Y. Swan) \textsc{Mathematics department,  Universit\'e
    de Li\`ege, Li\`ege, Belgium}

(E. Azmoodeh) \textsc{Department of Mathematics and Statistics, University of
  Vaasa, Finland}

(G. Poly) \textsc{Institut de Recherche Math\'ematiques de Rennes,
  Universit\'e de Rennes 1, Rennes, France}

\

\emph{E-mail address}, B. Arras {\tt barras@ulg.ac.be }

\emph{E-mail address}, E. Azmoodeh {\tt ehsan.azmoodeh@uva.fi }

\emph{E-mail address}, G. Poly {\tt guillaume.poly@univ-rennes1.fr }

\emph{E-mail address}, Y. Swan  {\tt yswan@ulg.ac.be }

\end{document}